\documentclass{article}

\usepackage{fullpage}
\usepackage[english]{babel}
\usepackage[utf8]{inputenc}
\usepackage{amsmath,amsfonts,amssymb,amsthm}
\usepackage{enumitem}
\usepackage[all]{xy}
\usepackage{graphicx}
\usepackage[usenames, dvipsnames]{xcolor}
\usepackage{tikz}
\usetikzlibrary{arrows.meta}

\newcommand{\NN}{\mathbb{N}}
\newcommand{\NNo}{\NN_0}
\newcommand{\ZZ}{\mathbb{Z}}
\newcommand{\RR}{\mathbb{R}}

\newcommand{\eps}{\varepsilon}

\newcommand{\G}{\Gamma}
\newcommand{\X}{\mathcal{X}}

\newcommand{\Sn}{\mathcal{S}_n}
\newcommand{\Sq}{\Sn^{(q)}}

\newcommand{\A}{\mathcal{A}}
\newcommand{\T}{\mathcal{T}}

\newcommand{\gp}[1]{\langle#1\rangle}
\newcommand{\ov}[1]{\overline{#1}}

\newcommand{\wt}[1]{\widetilde{#1}}

\theoremstyle{definition}
\newtheorem{de}{\normalsize Definition}[section]
\newtheorem{lem}[de]{\normalsize Lemma}
\newtheorem{prop}[de]{\normalsize Proposition}
\newtheorem{theo}[de]{\normalsize Theorem}
\newtheorem{cor}[de]{\normalsize Corollary}

\newtheorem{innercustomcor}{Corollary}
\newenvironment{customcor}[1]
{\renewcommand\theinnercustomcor{#1}\innercustomcor}
{\endinnercustomcor}

\newtheorem{innercustomthm}{Theorem}
\newenvironment{customthm}[1]
{\renewcommand\theinnercustomthm{#1}\innercustomthm}
{\endinnercustomthm}

\makeatletter
\newcommand{\subjclass}[2][1991]{%
	\let\@oldtitle\@title%
	\gdef\@title{\@oldtitle\footnotetext{#1 \emph{Mathematics subject classification.} #2.}}%
}
\newcommand{\keywords}[1]{%
	\let\@@oldtitle\@title%
	\gdef\@title{\@@oldtitle\footnotetext{\emph{Key words and phrases.} #1.}}%
}
\makeatother

\author{Boris Colombari}

\title{A diagrammatical characterization of Milnor invariants}
\date{}

\begin{document}
\maketitle

\begin{abstract}
The goal of this paper is to give a diagrammatical characterization of the information given by the Milnor invariants of links and string links. More precisely, we describe when two string links have equal Milnor invariants of length $\leq q$ and when a link has trivial Milnor invariants of lenght $\leq q$. This will be done through the use of welded knot theory, involving the notions of arrow calculus and $w_q$--concordance introduced by J-B. Meilhan and A. Yasuhara. These results is to be compared to the classification of links up to $C_q$--concordance obtained by J. Conant, R. Schneiderman and P. Teichner.
\end{abstract}

\section*{Introduction}

In 1954, J. Milnor defined the now well-known Milnor invariants of a link in~\cite{LGMil}, encoding the information given by the longitudes of the link in the coefficients of some formal power series. These longitudes are part of the peripheral system of the link, which is composed of its fundamental group, meridians and longitudes. Although topological in nature, the peripheral system can be described in a combinatorial way using the Wirtinger presentation of the group for a diagram representing the link. This is the approach that will be used in this paper, as we will deal with objects that are defined diagrammatically. \\

The goal of this paper is to give some new insights on the type of information the Milnor invariants are able to detect, which is related to the study of knotted objects in $4$ dimensions. Indeed the characterization will be made through the use of welded objects, which can be seen as a diagrammatical intermediary between classical knots and (the ribbon subclass of) knotted surfaces in $\RR^4$. In \cite{RTwSL}, B. Audoux, P. Bellingeri, J-B. Meilhan and E. Wagner gave a diagrammatical characterization of the non-repeating Milnor invariants in terms of link-homotopy on welded string links, which consists in allowing strands to cross themselves. They proved that two string links have equal non-repeating Milnor invariants if and only if they are link-homotopic as welded string links. In this paper we obtain similar diagrammatical characterizations of the general Milnor invariants of length $\leq q$, using the notion of $w_q$--concordance on welded string links. In~\cite{CRPSL}, B. Audoux and J-B. Meilhan also gave a characterization of the reduced peripheral system of links in terms of link-homotopy, and we obtain a similar result for the $q$--nilpotent peripheral system of links in terms of $w_q$--concordance, leading to a diagrammatical characterization of links with trivial Milnor invariants of length $\leq q$. \\

The main tool in this paper is the arrow calculus developped by J-B. Meilhan and A. Yasuhara in~\cite{MeiYas}, which describes surgeries on virtual diagrams by unitrivalent trees, giving an analogue on welded objects of the notion of clasper calculus introduced on classical links by K. Habiro in~\cite{Hab}. The notion of $w_q$--equivalence consists in allowing surgeries along oriented trees of degree $\geq q$, defining a finer equivalence relation on welded objects as $q$ increases. When combined with welded concordance, we obtain the notion of $w_q$--concordance, which we will use to characterize the Milnor invariants of length $\leq q$ on welded objects. More precisely, we obtain the following result:

\begin{customthm}{1}
The Milnor invariants of length $\leq q$ are complete invariants of $w_q$--concordance on welded string links.
\end{customthm}

As classical (string) links inject into welded ones (see Theorem 1.B of~\cite{GouPolVir} for the case of virtual knots, which extends to welded objects), this gives a diagrammatical characterization for classical string links with equal Milnor invariants:

\begin{customcor}{1}
Two classical string links have equal Milnor invariants of length $\leq q$ if and only if they are $w_q$--concordant when considered as welded string links.
\end{customcor}

As proven by A. Casejuane and J-B. Meilhan in the upcoming paper \cite{FTIwSLRT} following~\cite{Cas21}, for $q\leq 3$ welded string links are classified up to $w_q$--equivalence by their finite type invariants of order $<q$, and it is conjectured to still be true for $q\geq 4$. Theorem 1 shows that among these finite type invariants, the Milnor invariants are able to distinguish welded string links up to $w_q$--concordance. The others invariants, namely the closure invariants which are extracted from the normalized Alexander polynomials of the welded long knots obtained by putting some of the components end to end, are needed to distinguish between two welded string links which are concordant but not $w_q$--equivalent. \\

Moreover, combining Theorem 1 with results from~\cite{MeiYas} and~\cite{Cas21} we have the following corollary on the finite type concordance invariants of welded string links:

\begin{customcor}{2}
For two welded string links $L$ and $L'$, the following statements are equivalent:
\begin{enumerate}
\item $L$ and $L'$ are $w_q$--concordant,
\item $L$ and $L'$ have the same Milnor invariants of length $\leq q$,
\item $L$ and $L'$ the same finite type concordance invariants of degree $<q$.
\end{enumerate}
\end{customcor}

This can be seen as a welded analogue of the result obtained by N. Habegger and G. Masbaum in~\cite{TKIMI}, stating that all rational finite type concordance invariants of string links are given by the Milnor invariants. \\

We also obtain a classification result similar to Theorem 1 on welded links, but in terms of $q$--nilpotent peripheral systems instead of Milnor invariants. More precisely, we use the equivalence class of $q$--nilpotent peripheral systems under a natural equivalence relation which removes the ambiguity coming from the choice of meridians (see Section~\ref{Sec31}).

\begin{customthm}{2}
The equivalence class of $q$--nilpotent peripheral system is a complete invariant of $w_q$--concordance on welded links.
\end{customthm}

The Milnor invariants can still be used to detect the unlink modulo $w_q$--equivalence:

\begin{customcor}{3}
A welded link is $w_q$--concordant to the unlink if and only if it has vanishing Milnor invariants of length $\leq q$.
\end{customcor}

As in the case of string links, we obtain a diagrammatical charaterization for classical links with trivial Milnor invariants:

\begin{customcor}{4}
A classical link has vanishing Milnor invariants of length $\leq q$ if and only if it is $w_q$--concordant to the unlink when considered as a welded link.
\end{customcor}

This final result is to be compared to the classification of classical links up to $C_q$--concordance in terms of Milnor invariants and higher order Sato-Levine and Arf invariants obtained by J. Conant, R. Schneiderman and P. Teichner in~\cite{ConSchTei}. The notion of $C_q$--concordance is obtained by combining $C_q$--equivalence, which comes from clasper surgeries described by K. Habiro in~\cite{Hab}, and concordance. This construction inspired that of $w_q$--concordance through arrow calculus on welded objects. To illustrate the role of these additional invariants, we give an example at the end of the paper of a link which is $w_3$--concordant to the unlink as a welded link, but not $C_3$--concordant to it. \\

This paper is divided in three main sections. Section~\ref{Sec1} is dedicated to the introduction of the objects which will be used throughout the paper, starting with some algebraic properties of commutators in Section~\ref{Sec11}. In Section~\ref{Sec12}, we define welded links and string links, give their representation by virtual diagrams and then by arrow and tree presentations coming from arrow calculus. In Section~\ref{Sec13}, we define the notion of concordance and $w_q$--concordance for welded objects. Section~\ref{Sec2} is dedicated to the classification of welded string links up to $w_q$--concordance. It starts with the introduction of the Milnor invariants, then we prove that any tree presentation of a welded string link is $w_q$--concordant to an ascending presentation, which is a prerequisite to the proof of Theorem 1. Section~\ref{Sec3} is dedicated to the classification of welded links up to $w_q$--concordance, starting with the definition and properties of $q$--nilpotent peripheral systems in Section~\ref{Sec31}, then proving that any tree presentation of a welded link is $w_q$--equivalent to a sorted presentation and giving a classification of these presentations in Section~\ref{Sec32}. \\

\textit{Acknowledgements.} This article summarizes some results obtained during my PhD thesis. I would like to thank my thesis tutor B. Audoux for his guidance and helpful advice in the writing of this paper, as well as J-B. Meilhan, A. Yasuhara and M. Chrisman for their insightful comments on a previous version of this paper. I am also grateful to l'Institut de Mathématiques de Marseille for hosting me during the redaction of my PhD thesis, which was supported by l'\'{E}cole Normale Supérieure de Cachan.

\section{Notations and settings}\label{Sec1}

\subsection{Algebraic notions}\label{Sec11}

For elements $a,b$ in a group $G$, we denote by $a^b:=b^{-1}ab$ the \emph{conjugate of $a$ by $b$}, and $[a,b]:=a^{-1}b^{-1}ab$ the \emph{commutator of $a$ and $b$}. For subsets $A,B$ of $G$, we denote by $[A,B]$ the normal subgroup of $G$ generated by elements $[a,b]$ for $a\in A$ and $b\in B$. \\

The \emph{lower central series of $G$}, denoted by $(\G_qG)_{q\geq 1}$,  is defined by induction on $q$ by $\G_1G:=G$ and $\G_{q+1}G:=[G,\G_qG]$ for $q\geq 1$. The $\G_qG$ are normal subgroups of $G$, and we denote by $N_qG:=G/\G_qG$ the \emph{nilpotent quotients of $G$}. The subgroup $\G_qG$ contains all iterated commutators of weight $\geq q$, where the weight corresponds to the number of elements. For example, $[[g_1,g_2],g_3]$ is a commutator of weight $3$, and $[[g_1,g_2],[g_3,g_4]]$ and $[g_1,[[g_2,g_3],g_4]]$ are commutators of weight $4$. The classification of ascending and sorted presentations in Sections~\ref{Sec2} and \ref{Sec3} will require the following description of elements of $\G_qG$, which follows from well known identites given in \cite{MKS}:

\begin{prop}\label{gencom}
If a group $G$ is generated by elements $g_1,\ldots,g_n$, any element of $\G_qG$ can be written as a product of commutators of weight $\geq q$ on $g_1^{\pm 1},\ldots,g_n^{\pm 1}$.
\end{prop}

For $n\geq 1$, let $F_n$ denote the free group on $n$ generators $m_1,\ldots,m_n$. \\
%For $1\leq i\leq n$, we denote by $m_i^{\ast}:F_n\to\ZZ$ the group homomorphism sending $x_i$ to $1$ and the $x_j$'s to $0$ for $j\neq i$. For $q\geq 2$, this homomorphism factors through the projection $F_n\to N_qF_n$, and we denote again by $x_i^{\ast}:N_qF_n\to\ZZ$ the induced homomorphism.
We denote by $\Sn:=\ZZ\gp{\gp{X_1,\ldots,X_n}}$ the $\ZZ$--algebra of formal power series with $\ZZ$ coefficients on $n$ non-commutating variables. We denote by $U_1(\Sn)$ the set of power series with constant term equal to $1$, which is a subgroup of the multiplicative group of units of $\Sn$. \\
We denote by $\X$ the two-sided ideal of $\Sn$ of power series whose constant term is zero. For $q\geq 1$, $\X^q$ is the two-sided ideal of power series whose monomials are of degree greater than or equal to $q$. We denote by $\Sq:=\Sn/\X^q$ the $\ZZ$--algebra obtained from $\Sn$ where all monomials of degree at least $q$ are made equal to zero. We denote by $U_1(\Sq)$ the subgroup of units of $\Sq$ with contant term equal to $1$. Note that we have a canonical isomorphism $U_1(\Sq)\simeq U_1(\Sn)/(1+\X^q)$. \\

We will consider the \emph{Magnus expansion} $E:F_n\to U_1(\Sn)$, defined as the group homomorphism sending $m_i$ to $1+X_i$. It is easily checked (see Section 5.5 of \cite{MKS}) that $E$ is injective, so $E:F_n\to\text{Im}(E)$ is an isomorphism. Moreover, it follows from results in Section 5.7 of \cite{MKS} that $E(\G_qF_n)=\G_q\text{Im}(E)=\text{Im}(E)\cap(1+\X^q)$. As a result, we have:

\begin{prop}\label{injMagexp}
For $q\geq 1$, the Magnus expansion induces a homomorphism $E_q:N_qF_n\to U_1(\Sq)$, which is injective.
\end{prop}

We will use this injective homomorphism $E_q$ to extract the Milnor invariants from the longitudes of welded objects, and conversely to recover the longitudes from the Milnor invariants.

\subsection{Representations of welded links and welded string links}\label{Sec12}

\subsubsection{Virtual diagrams}\label{Sec121}

In what follows, we fix a positive integer $n$, and $n$ points $p_1,\ldots,p_n$ in $(0,1)$ such that $<p_1<\ldots<p_n$.

\begin{de}\label{vrtdiag}
An \emph{$n$--component virtual link diagram} is an immersion $f:\{1,\ldots,n\}\times S^1\to\RR^2$ whose singularities are transverse double points labeled with an element of $\{+,-,v\}$. An \emph{$n$--component virtual string link diagram} is a proper immersion $f:\{1,\ldots,n\}\times[0,1]\to[0,1]\times[0,1]$ such that $f(i,\eps)=(p_i,\eps)$ for $1\leq i\leq n$ and $\eps=0,1$, and whose singularities are transverse double points labeled with an element of $\{+,-,v\}$.
\end{de}

We denote by $D$ such a diagram, and $D_i$ its $i^{th}$ component, given by $f(i,\cdot)$. This component is given the canonical orientation from $S^1$ for a link, or from $[0,1]$ for a string link. A double point is called a positive (resp. negative, resp. virtual) crossing if it is labeled with a $+$ (resp. $-$, resp. $v$). A positive or negative crossing is called classical, and is represented by erasing part of one of the strands (depending on the sign) at the crossing, while a virtual crossing is represented by drawing a circle around it :

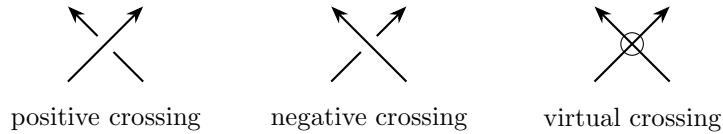
\begin{center}
\begin{tikzpicture}
\draw [thick,-Stealth] (1,0) -- (0,1) ;
\fill [white] (0.5,0.5) circle (0.15) ;
\draw [thick,-Stealth] (0,0) -- (1,1) ;
\draw (0.5,-0.5) node{positive crossing} ;

\begin{scope}[xshift=3.5cm]
\draw [thick,-Stealth] (0,0) -- (1,1) ;
\fill [white] (0.5,0.5) circle (0.15) ;
\draw [thick,-Stealth] (1,0) -- (0,1) ;
\draw (0.5,-0.5) node{negative crossing} ;
\end{scope}

\begin{scope}[xshift=7cm]
\draw [thick,-Stealth] (0,0) -- (1,1) ;
\draw (0.5,0.5) circle (0.15) ;
\draw [thick,-Stealth] (1,0) -- (0,1) ;
\draw (0.5,-0.5) node{virtual crossing} ;
\end{scope}
\end{tikzpicture}
\end{center}

At a classical crossing, the strand that has been partly erased is called the \emph{lower strand}, and the other one is called the \emph{upper strand}. The preimage of the double point on the lower strand is called an \emph{undercrossing}, and the one on the upper strand is called an \emph{overcrossing}. We call \emph{arcs of $D$} the portions of strands delimited by the undercrossings.

%We denote by $vLD_n$ (resp. $vSLD_n$) the set of $n$--component virtual link (resp. string link) diagrams, up to isotopy and reparametrization.

\begin{de}\label{defwelded}
An \emph{$n$--component welded link} (resp. \emph{welded string link}) is an equivalence class of $n$--component virtual link (resp. string link) diagrams up to the local moves, and their left-right mirror images, displayed in Figure~\ref{Figwldmoves}.
%We denote by $w\Ln$ (resp. $w\SL$) the set of $n$--component welded links (resp. welded string links).

\begin{figure}
\centering
\begin{tikzpicture}
\draw [thick] (0,0) -- (0,1) ;
\draw [<->] (0.5,0.5) -- (1.3,0.5) ;
\draw (0.9,0.8) node{R1} ;
\draw [thick] (1.8,0) .. controls +(0,0.1) and +(-0.1,-0.2) .. (1.9,0.5) ;
\draw [thick] (1.9,0.5) .. controls +(0.2,0.4) and +(0,0.3) .. (2.3,0.5) ;
\fill [white] (1.9,0.5) circle (0.15) ;
\draw [thick] (2.3,0.5) .. controls +(0,-0.3) and +(0.2,-0.4) .. (1.9,0.5) ;
\draw [thick] (1.9,0.5) .. controls +(-0.1,0.2) and +(0,-0.1) .. (1.8,1) ;

\begin{scope}[xshift=3.7cm]
\draw [thick] (0,0) -- (0,1) ;
\draw [thick] (0.5,0) -- (0.5,1) ;
\draw [<->] (1,0.5) -- (1.8,0.5) ;
\draw (1.4,0.8) node{R2} ;
\draw [thick] (2.8,0) .. controls +(-0.2,0.1) and +(0,-0.2) .. (2.4,0.5) ;
\draw [thick] (2.4,0.5) .. controls +(0,0.2) and +(-0.2,-0.1) .. (2.8,1) ;
\fill [white] (2.55,0.8) circle (0.15) ;
\fill [white] (2.55,0.2) circle (0.15) ;
\draw [thick] (2.3,0) .. controls +(0.2,0.1) and +(0,-0.2) .. (2.7,0.5) ;
\draw [thick] (2.7,0.5) .. controls +(0,0.2) and +(0.2,-0.1) .. (2.3,1) ;
\end{scope}

\begin{scope}[xshift=7.9cm]
\draw [thick] (0.99,0) -- ++(-0.58,1);
\fill [white] (0.58,0.71) circle (0.12) ;
\draw [thick] (0.17,0) -- ++(0.58,1);
\fill [white] (0.34,0.29) circle (0.12) ;
\fill [white] (0.82,0.29) circle (0.12) ;
\draw [thick] (0,0.29) -- (1.16,0.29) ;
\draw [<->] (1.66,0.5) -- (2.46,0.5) ;
\draw (2.06,0.8) node{R3} ;
\draw [thick] (3.13,1) -- ++(0.58,-1);
\fill [white] (3.54,0.29) circle (0.12) ;
\draw [thick] (3.95,1) -- ++(-0.58,-1);
\fill [white] (3.78,0.71) circle (0.12) ;
\fill [white] (3.3,0.71) circle (0.12) ;
\draw [thick] (2.96,0.71) -- (4.12,0.71) ;
\end{scope}

\begin{scope}[yshift=-2cm]
\draw [thick] (0,0) -- (0,1) ;
\draw [<->] (0.5,0.5) -- (1.3,0.5) ;
\draw (0.9,0.8) node{vR1} ;
\draw [thick] (1.8,0) .. controls +(0,0.1) and +(-0.1,-0.2) .. (1.9,0.5) ;
\draw [thick] (1.9,0.5) .. controls +(0.2,0.4) and +(0,0.3) .. (2.3,0.5) ;
\draw (1.9,0.5) circle (0.15) ;
\draw [thick] (2.3,0.5) .. controls +(0,-0.3) and +(0.2,-0.4) .. (1.9,0.5) ;
\draw [thick] (1.9,0.5) .. controls +(-0.1,0.2) and +(0,-0.1) .. (1.8,1) ;
\end{scope}

\begin{scope}[xshift=3.7cm,yshift=-2cm]
\draw [thick] (0,0) -- (0,1) ;
\draw [thick] (0.5,0) -- (0.5,1) ;
\draw [<->] (1,0.5) -- (1.8,0.5) ;
\draw (1.4,0.8) node{vR2} ;
\draw [thick] (2.8,0) .. controls +(-0.2,0.1) and +(0,-0.2) .. (2.4,0.5) ;
\draw [thick] (2.4,0.5) .. controls +(0,0.2) and +(-0.2,-0.1) .. (2.8,1) ;
\draw (2.55,0.8) circle (0.15) ;
\draw (2.55,0.2) circle (0.15) ;
\draw [thick] (2.3,0) .. controls +(0.2,0.1) and +(0,-0.2) .. (2.7,0.5) ;
\draw [thick] (2.7,0.5) .. controls +(0,0.2) and +(0.2,-0.1) .. (2.3,1) ;
\end{scope}

\begin{scope}[xshift=7.9cm,yshift=-2cm]
\draw [thick] (0,0.29) -- (1.16,0.29) ;
\draw [thick] (0.17,0) -- ++(0.58,1);
\draw [thick] (0.99,0) -- ++(-0.58,1);
\draw (0.34,0.29) circle (0.12) ;
\draw (0.82,0.29) circle (0.12) ;
\draw (0.58,0.71) circle (0.12) ;
\draw [<->] (1.66,0.5) -- (2.46,0.5) ;
\draw (2.06,0.8) node{vR3} ;
\draw [thick] (2.96,0.71) -- (4.12,0.71) ;
\draw [thick] (3.95,1) -- ++(-0.58,-1);
\draw [thick] (3.13,1) -- ++(0.58,-1);
\draw (3.78,0.71) circle (0.12) ;
\draw (3.3,0.71) circle (0.12) ;
\draw (3.54,0.29) circle (0.12) ;
\end{scope}

\begin{scope}[xshift=1.15cm,yshift=-4cm]
\draw [thick] (0,0.29) -- (1.16,0.29) ;
\draw [thick] (0.99,0) -- ++(-0.58,1);
\draw [white,fill=white] (0.58,0.71) circle (0.12) ;
\draw [thick] (0.17,0) -- ++(0.58,1);
\draw (0.34,0.29) circle (0.12) ;
\draw (0.82,0.29) circle (0.12) ;
\draw [<->] (1.66,0.5) -- (2.46,0.5) ;
\draw (2.06,0.8) node{Mixed} ;
\draw [thick] (2.96,0.71) -- (4.12,0.71) ;
\draw [thick] (3.13,1) -- ++(0.58,-1);
\draw [white,fill=white] (3.54,0.29) circle (0.12) ;
\draw [thick] (3.95,1) -- ++(-0.58,-1);
\draw (3.78,0.71) circle (0.12) ;
\draw (3.3,0.71) circle (0.12) ;
\end{scope}

\begin{scope}[xshift=6.75cm,yshift=-4cm]
\draw [thick] (0.99,0) -- ++(-0.58,1);
\draw (0.58,0.71) circle (0.12) ;
\draw [thick] (0.17,0) -- ++(0.58,1);
\draw [white,fill=white] (0.34,0.29) circle (0.12) ;
\draw [white,fill=white] (0.82,0.29) circle (0.12) ;
\draw [thick] (0,0.29) -- (1.16,0.29) ;
\draw [<->] (1.66,0.5) -- (2.46,0.5) ;
\draw (2.06,0.8) node{OC} ;
\draw [thick] (3.13,1) -- ++(0.58,-1);
\draw (3.54,0.29) circle (0.12) ;
\draw [thick] (3.95,1) -- ++(-0.58,-1);
\draw [white,fill=white] (3.78,0.71) circle (0.12) ;
\draw [white,fill=white] (3.3,0.71) circle (0.12) ;
\draw [thick] (2.96,0.71) -- (4.12,0.71) ;
\end{scope}
\end{tikzpicture}
\caption{Welded Reidemeister moves.}\label{Figwldmoves}
\end{figure}
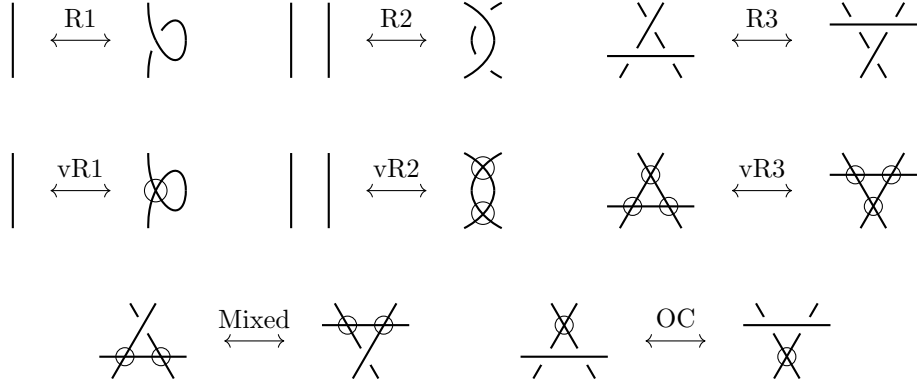
\end{de}

\begin{de}\label{GDgroup}
We associate a group $G_D$ to a virtual diagram $D$, defined by a presentation similar to the Wirtinger presentation of the fundamental group for classical links. More precisely, we associate one generator to each arc of $D$, and one relation $c=b^{a^{\eps}}$ for each classical crossing, where $a$ is the generator associated to the upper strand, $b$ (resp. $c$) is the generator associated to the lower strand preceding (resp. following) the crossing, and $\eps$ is the sign of the crossing.
\end{de}

As we will often need to consider the product of the conjugating elements involved in such relations along a subinterval of a component, we introduce the following notation:

\begin{de}\label{subintword}
For a subinterval $J$ of a component of a virtual diagram $D$, we denote by $w_J$ the word on the generators of $G_D$ and their inverse obtained by taking the product of the generators associated to the overstrands at the undercrossings on $J$, in the order given by the orientation of the diagram.
\end{de}

We will make use of such words throughout this paper, sometimes as elements of $G_D$ or $N_qG_D$, and sometimes strictly as words on the generators and their inverse, with no relations between them. For example, on the diagram illustrated below we have $w_J=ab^{-1}bc$.

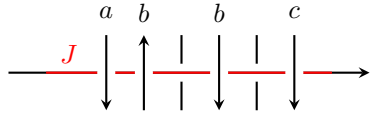
\begin{center}
\begin{tikzpicture}
\draw [thick,] (2,0.5) -- (2,-0.5) ;
\fill [white] (2,0) circle (0.12) ;
\draw [thick,] (3,0.5) -- (3,-0.5) ;
\fill [white] (3,0) circle (0.12) ;
\draw [thick,-stealth] (-0.3,0) -- (4.5,0) ;
\draw [thick,red] (0.2,0) -- (4,0) ;
\draw [red] (0.5,0.25) node{$J$} ;
\fill [white] (1,0) circle (0.12) ;
\draw [thick,-stealth] (1,0.5) -- (1,-0.5) ;
\draw (1,0.8) node{$a$} ;
\fill [white] (1.5,0) circle (0.12) ;
\draw [thick,-stealth] (1.5,-0.5) -- (1.5,0.5) ;
\draw (1.5,0.8) node{$b$} ;
\fill [white] (2.5,0) circle (0.12) ;
\draw [thick,-stealth] (2.5,0.5) -- (2.5,-0.5) ;
\draw (2.5,0.8) node{$b$} ;
\fill [white] (3.5,0) circle (0.12) ;
\draw [thick,-stealth] (3.5,0.5) -- (3.5,-0.5) ;
\draw (3.5,0.8) node{$c$} ;
\end{tikzpicture}
\end{center}

For a virtual diagram $D$, we denote by $a_{ij}$ the generators of $G_D$ associated to the arcs of $D_i$. Since the relations in the definition of $G_D$ are conjugating relations between generators associated to the same component, the homomorphisms $a_i^{\ast}:\gp{a_{ij}}\to\ZZ$ defined by $a_i^{\ast}(a_{kj})=\delta_{ik}$ induce homomorphisms $a_i^{\ast}:G_D\to\ZZ$. \\

Let $D$ be a virtual diagram of a welded link. For a given point $p_i$ on an arc of $D_i$ between two undercrossings, we denote by $m_i$ the generator of $G_D$ associated to this arc, and call it a \emph{meridian of $D_i$}. Let $D_i^{p_i}$ denote the subinterval $D_i\setminus\{p_i\}$ of $D_i$. We call \emph{longitude of $D_i$ associated to $m_i$} the element $\ell_i:=m_i^{-k}w_{D_i^{p_i}}\in G_D$, where $k:=a_i^{\ast}(w_{D_i^{p_i}})$. These notions of meridian and longitude come from the case of a classical link $L$ in $\RR^3$, where a meridian of a component $L_i$ is the boundary of a small disk which is positively transverse to $L_i$ in one point, whose corresponding element in $G_L:=\pi_1(L)$ is the generator associated to that point, and a longitude of $L_i$ is a loop which runs parallel to it and has linking number with $L_i$ equal to $0$. The correcting term $m_i^{-k}$ in $\ell_i$ corresponds to the linking number constraint of the classical case. A peripheral system of $D$ is given by its group $G_D$ and a choice of meridians and longitudes. Up to isomorphisms which conjugate each meridian and associated longitude by the same element, it is an invariant of the welded link represented by $D$. \\

We can also consider the notion of peripheral system for a virtual diagram $D$ of a welded string link. In this case, there is a natural choice for the meridians: $m_i$ is the generator associated to the bottom arc of $D_i$, and the \textquotedblleft longitude" $\ell_i$ is defined as $\ell_i:=m_i^{-k}w_{D_i}$, with $k:=a_i^{\ast}(w_{D_i})$. With these notation, the generator associated to the top arc of $D_i$ is equal to $m_i^{\ell_i}$ in $G_D$.

\subsubsection{Arrow calculus}\label{Sec122}

In \cite{MeiYas}, J-B. Meilhan and A. Yasuhara developed the notion of \emph{arrow calculus}, inspired by clasper calculus on classical objects, which consists in adding arrows, called w-arrows, on virtual diagrams representing some surgeries as indicated below (note that in \cite{MeiYas}, the virtual crossings are represented without the small circle):

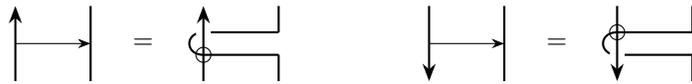
\begin{center}
\begin{tikzpicture}
\draw [thick,-Stealth] (0,0) -- (0,1) ;
\draw [thick] (1,0) -- (1,1) ;
\draw [-Stealth] (0,0.5) -- (1,0.5) ;
\draw (1.7,0.5) node{$=$} ;
\draw [thick] (3.5,0) -- (3.5,0.35) ;
\draw [thick] (3.5,0.35) -- (2.5,0.35) ;
\draw [thick] (2.5,0.35) .. controls +(-0.25,0) and +(-0.25,0) .. (2.5,0.65) ;
\draw [thick] (2.5,0.65) -- (3.5,0.65) ;
\draw [thick] (3.5,0.65) -- (3.5,1) ;
\fill [white] (2.5,0.65) circle (0.1) ;
\draw (2.5,0.35) circle (0.1) ;
\draw [thick,-Stealth] (2.5,0) -- (2.5,1) ;

\begin{scope}[xshift=5.5cm]
\draw [thick,-Stealth] (0,1) -- (0,0) ;
\draw [thick] (1,0) -- (1,1) ;
\draw [-Stealth] (0,0.5) -- (1,0.5) ;
\draw (1.7,0.5) node{$=$} ;
\draw [thick] (3.5,0) -- (3.5,0.35) ;
\draw [thick] (3.5,0.35) -- (2.5,0.35) ;
\draw [thick] (2.5,0.35) .. controls +(-0.25,0) and +(-0.25,0) .. (2.5,0.65) ;
\draw [thick] (2.5,0.65) -- (3.5,0.65) ;
\draw [thick] (3.5,0.65) -- (3.5,1) ;
\fill [white] (2.5,0.35) circle (0.1) ;
\draw (2.5,0.65) circle (0.1) ;
\draw [thick,-Stealth] (2.5,1) -- (2.5,0) ;
\end{scope}
\end{tikzpicture}
\end{center}

Note that the relative position of the classical and virtual crossings involved in the surgery depend on the orientation of the strand containing the tail of the w-arrow. Moreover, the side of the strand from which the w-arrow is attached is also relevant. A w-arrow may cross other strands of the virtual diagram or other w-arrows, in which case virtual crossings are created when performing the surgery. \\

A w-arrow may also contain a dot, representing a twist, which modifies the surgery as indicated below. A w-arrow may contain multiple dots (inducing as many twists), but since the twists can by cancelled pairwise by using $vR2$ moves, we may assume that each w-arrow contains at most one dot.

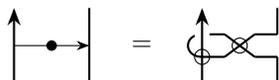
\begin{center}
\begin{tikzpicture}
\draw [thick,-Stealth] (0,0) -- (0,1) ;
\draw [thick] (1,0) -- (1,1) ;
\draw [-Stealth] (0,0.5) -- (1,0.5) ;
\fill [black] (0.5,0.5) circle (0.07) ;
\draw (1.7,0.5) node{$=$} ;
\draw [thick] (3.5,0) -- (3.5,0.35) ;
\draw [thick] (3.5,0.35) -- (3.2,0.35) ;
\draw [thick] (3.2,0.35) -- (2.8,0.65) ;
\draw [thick] (2.8,0.65) -- (2.5,0.65) ;
\draw [thick] (2.5,0.65) .. controls +(-0.25,0) and +(-0.25,0) .. (2.5,0.35) ;
\draw [thick] (2.5,0.35) -- (2.8,0.35) ;
\draw [thick] (2.8,0.35) -- (3.2,0.65) ;
\draw [thick] (3.2,0.65) -- (3.5,0.65) ;
\draw [thick] (3.5,0.65) -- (3.5,1) ;
\fill [white] (2.5,0.65) circle (0.1) ;
\draw (2.5,0.35) circle (0.1) ;
\draw (3,0.5) circle (0.1) ;
\draw [thick,-Stealth] (2.5,0) -- (2.5,1) ;
\end{tikzpicture}
\end{center}

\begin{de}\label{arrowpres}
An \emph{arrow presentation} of a virtual diagram $D$ is a pair $(V,\A)$, where $V$ is a virtual diagram with no classical crossing and $\A$ is a family of w-arrows on $V$ such that the diagram obtained from $V$ by performing the surgeries corresponding to $\A$ is $D$. Two arrow presentations are called \emph{equivalent} if they give equivalent virtual diagrams.
\end{de}

As noted in \cite{MeiYas}, any classical crossing on a virtual diagram can be replaced by a virtual crossing and a w-arrow next to it:

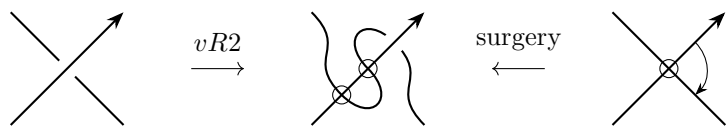
\begin{center}
\begin{tikzpicture}
\draw [thick] (0,1.5) -- (1.5,0) ;
\fill [white] (0.75,0.75) circle (0.15) ;
\draw [thick,-Stealth] (0,0) -- (1.5,1.5) ;
\draw (2.75,1.1) node{$vR2$} ;
\draw [->] (2.4,0.75) -- (3.1,0.75) ;

\begin{scope}[xshift=4cm]
\draw [thick] (1.5,0) .. controls +(-0.5,0.5) and +(0.5,-0.5) .. (1.1,1.1) ;
\draw [thick] (1.1,1.1) .. controls +(-0.4,0.4) and +(-0.4,0.4) .. (0.75,0.75) ;
\draw [thick] (0.75,0.75) .. controls +(0.4,-0.4) and +(0.4,-0.4) .. (0.4,0.4) ;
\draw [thick] (0.4,0.4) .. controls +(-0.5,0.5) and +(0.5,-0.5) .. (0,1.5) ;
\fill [white] (1.1,1.1) circle (0.15) ;
\draw (0.75,0.75) circle (0.12) ;
\draw (0.4,0.4) circle (0.12) ;
\draw [thick,-Stealth] (0,0) -- (1.5,1.5) ;
\draw (2.75,1.1) node{surgery} ;
\draw [<-] (2.4,0.75) -- (3.1,0.75) ;
\end{scope}

\begin{scope}[xshift=8cm]
\draw [thick] (0,1.5) -- (1.5,0) ;
\draw (0.75,0.75) circle (0.12) ;
\draw [thick,-Stealth] (0,0) -- (1.5,1.5) ;
\draw [-Stealth] (1.1,1.1) .. controls +(0.2,-0.2) and +(0.2,0.2) .. (1.1,0.4) ;
\end{scope}
\end{tikzpicture}
\end{center}

As a result, any virtual diagram admits an arrow presentation. In Section 4 of~\cite{MeiYas}, Meilhan and Yasuhara give a list of local moves on w-arrows which result in equivalent arrow presentations. First, the extremities of w-arrows can pass through virtual crossings, but in general not through classical crossings or other w-arrow extremities. A twist can be moved along a w-arrow, passing through the strands or w-arrows that it intersects.

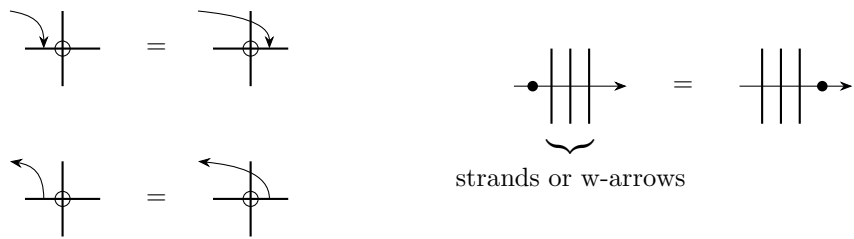
\begin{center}
\begin{tikzpicture}
\draw [thick] (-0.5,0) -- (0.5,0) ;
\draw [thick] (0,-0.5) -- (0,0.5) ;
\draw (0,0) circle (0.1) ;
\draw [-Stealth] (-0.7,0.5) .. controls +(0.3,-0.05) and +(0,0.35) .. (-0.25,0) ;
\draw (1.25,0) node{$=$} ;
\draw [thick] (2,0) -- (3,0) ;
\draw [thick] (2.5,-0.5) -- (2.5,0.5) ;
\draw (2.5,0) circle (0.1) ;
\draw [-Stealth] (1.8,0.5) .. controls +(0.5,-0.05) and +(0,0.35) .. (2.75,0) ;

\begin{scope}[yshift=-2cm]
\draw [thick] (-0.5,0) -- (0.5,0) ;
\draw [thick] (0,-0.5) -- (0,0.5) ;
\draw (0,0) circle (0.1) ;
\draw [Stealth-] (-0.7,0.5) .. controls +(0.3,-0.05) and +(0,0.35) .. (-0.25,0) ;
\draw (1.25,0) node{$=$} ;
\draw [thick] (2,0) -- (3,0) ;
\draw [thick] (2.5,-0.5) -- (2.5,0.5) ;
\draw (2.5,0) circle (0.1) ;
\draw [Stealth-] (1.8,0.5) .. controls +(0.5,-0.05) and +(0,0.35) .. (2.75,0) ;
\end{scope}

\begin{scope}[xshift=6cm,yshift=-0.5cm]
\draw [-Stealth] (0,0) -- (1.5,0) ;
\fill [black] (0.25,0) circle (0.07) ;
\draw [thick] (0.5,-0.5) -- (0.5,0.5) ;
\draw [thick] (0.75,-0.5) -- (0.75,0.5) ;
\draw [thick] (1,-0.5) -- (1,0.5) ;
\draw (0.75,-0.8) node{$\underbrace{}$} ;
\draw (0.75,-1.2) node{strands or w-arrows} ;
\draw (2.25,0) node{$=$} ;
\draw [-Stealth] (3,0) -- (4.5,0) ;
\draw [thick] (3.3,-0.5) -- (3.3,0.5) ;
\draw [thick] (3.55,-0.5) -- (3.55,0.5) ;
\draw [thick] (3.8,-0.5) -- (3.8,0.5) ;
\fill [black] (4.1,0) circle (0.07) ;
\end{scope}
\end{tikzpicture}
\end{center}

Changing the side of the strand from which the extremity of a w-arrow is attached has the following effect:

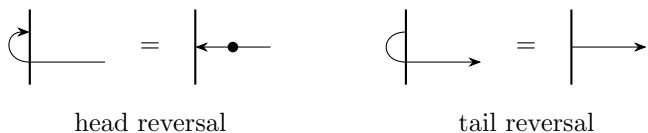
\begin{center}
\begin{tikzpicture}
\draw [thick] (0,0) -- (0,1) ;
\draw [-Stealth] (0,0.3) .. controls +(-0.35,0) and +(-0.35,0) .. (0,0.7) ;
\draw (0,0.3) -- (1,0.3) ;
\draw (1.6,0.5) node{$=$} ;
\draw [thick] (2.2,0) -- (2.2,1) ;
\draw [-Stealth] (3.2,0.5) -- (2.2,0.5) ;
\fill [black] (2.7,0.5) circle (0.07) ;
\draw (1.6,-0.5) node{head reversal} ;

\begin{scope}[xshift=5cm]
\draw [thick] (0,0) -- (0,1) ;
\draw (0,0.3) .. controls +(-0.35,0) and +(-0.35,0) .. (0,0.7) ;
\draw [-Stealth] (0,0.3) -- (1,0.3) ;
\draw (1.6,0.5) node{$=$} ;
\draw [thick] (2.2,0) -- (2.2,1) ;
\draw [-Stealth] (2.2,0.5) -- (3.2,0.5) ;
\draw (1.6,-0.5) node{tail reversal} ;
\end{scope}
\end{tikzpicture}
\end{center}

Finally, the four local moves illustrated on Figure~\ref{Figwarrowmoves} correspond to the classical Reidemeister moves and the $OC$ move, where a white dot indicates that the w-arrow may or may not contain a twist. \\

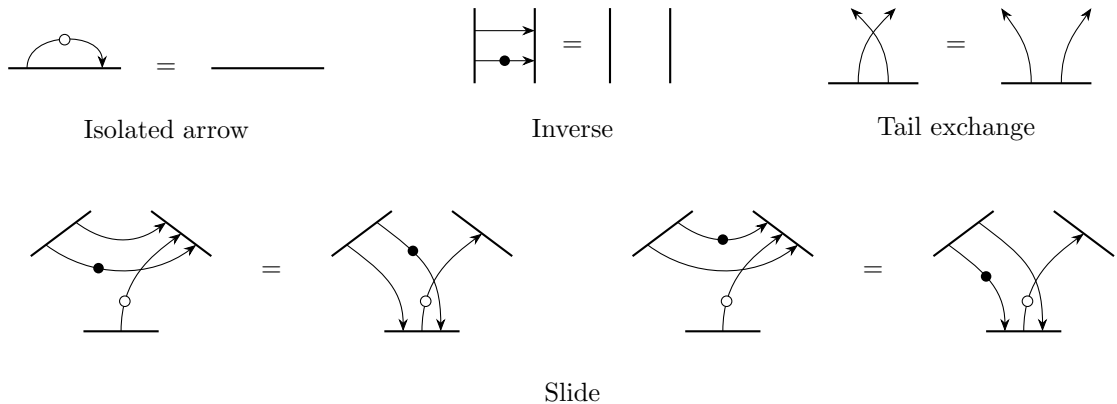
\begin{figure}
\centering
\begin{tikzpicture}
\begin{scope}[xshift=0.5cm]
\draw [thick] (0,0.2) -- (1.5,0.2) ;
\draw [-Stealth] (0.25,0.2) .. controls +(0,0.5) and +(0,0.5) .. (1.25,0.2) ;
\draw [fill=white] (0.75,0.58) circle (0.07) ;
\draw (2.1,0.2) node{$=$} ;
\draw [thick] (2.7,0.2) -- (4.2,0.2) ;
\draw (2.1,-0.6) node{Isolated arrow} ;
\end{scope}

\begin{scope}[xshift=8cm]
\draw [thick] (-1.3,0) -- (-1.3,1) ;
\draw [thick] (-0.5,0) -- (-0.5,1) ;
\draw [-Stealth] (-1.3,0.3) -- (-0.5,0.3) ;
\draw [-Stealth] (-1.3,0.7) -- (-0.5,0.7) ;
\fill [black] (-0.9,0.3) circle (0.07) ;
\draw (0,0.5) node{$=$} ;
\draw [thick] (0.5,0) -- (0.5,1) ;
\draw [thick] (1.3,0) -- (1.3,1) ;
\draw (0,-0.6) node{Inverse} ;
\end{scope}

\begin{scope}[xshift=13.1cm]
\draw [thick] (-1.7,0) -- (-0.5,0) ;
\draw [-Stealth] (-0.9,0) .. controls +(0,0.5) and +(0.3,-0.3) .. (-1.4,1) ;
\draw [-Stealth] (-1.3,0) .. controls +(0,0.5) and +(-0.3,-0.3) .. (-0.8,1) ;
\draw (0,0.5) node{$=$} ;
\draw [thick] (0.6,0) -- (1.8,0) ;
\draw [-Stealth] (1,0) .. controls +(0,0.5) and +(0.2,-0.3) .. (0.6,1) ;
\draw [-Stealth] (1.4,0) .. controls +(0,0.5) and +(-0.2,-0.3) .. (1.8,1) ;
\draw (0,-0.6) node{Tail exchange} ;
\end{scope}

\begin{scope}[xshift=2cm,yshift=-3.3cm]
\draw [thick] (-0.5,0) -- (0.5,0) ;
\draw [thick] (-1.2,1) -- (-0.4,1.6) ;
\draw [thick] (1.2,1) -- (0.4,1.6) ;
\draw [-Stealth] (0,0) .. controls +(0,0.7) and +(-0.4,-0.3) .. (0.8,1.3) ;
\draw [-Stealth] (-0.6,1.45) .. controls +(0.4,-0.3) and +(-0.4,-0.3) .. (0.6,1.45) ;
\draw [-Stealth] (-1,1.15) .. controls +(0.6,-0.45) and +(-0.6,-0.45) .. (1,1.15) ;
\draw [fill=white] (0.05,0.4) circle (0.07) ;
\fill [black] (-0.3,0.84) circle (0.07) ;
\draw (2,0.8) node{$=$} ;
\draw [thick] (3.5,0) -- (4.5,0) ;
\draw [thick] (2.8,1) -- (3.6,1.6) ;
\draw [thick] (5.2,1) -- (4.4,1.6) ;
\draw [-Stealth] (4,0) .. controls +(0,0.7) and +(-0.4,-0.3) .. (4.8,1.3) ;
\draw [-Stealth] (3.4,1.45) .. controls +(0.6,-0.45) and +(0,0.8) .. (4.25,0) ;
\draw [-Stealth] (3,1.15) .. controls +(0.4,-0.3) and +(0,0.6) .. (3.75,0) ;
\draw [fill=white] (4.05,0.4) circle (0.07) ;
\fill [black] (3.88,1.07) circle (0.07) ;
\draw (6,-0.8) node{Slide} ;
\end{scope}

\begin{scope}[xshift=10cm,yshift=-3.3cm]
\draw [thick] (-0.5,0) -- (0.5,0) ;
\draw [thick] (-1.2,1) -- (-0.4,1.6) ;
\draw [thick] (1.2,1) -- (0.4,1.6) ;
\draw [-Stealth] (0,0) .. controls +(0,0.7) and +(-0.4,-0.3) .. (0.8,1.3) ;
\draw [-Stealth] (-0.6,1.45) .. controls +(0.4,-0.3) and +(-0.4,-0.3) .. (0.6,1.45) ;
\draw [-Stealth] (-1,1.15) .. controls +(0.6,-0.45) and +(-0.6,-0.45) .. (1,1.15) ;
\draw [fill=white] (0.05,0.4) circle (0.07) ;
\fill [black] (0,1.22) circle (0.07) ;
\draw (2,0.8) node{$=$} ;
\draw [thick] (3.5,0) -- (4.5,0) ;
\draw [thick] (2.8,1) -- (3.6,1.6) ;
\draw [thick] (5.2,1) -- (4.4,1.6) ;
\draw [-Stealth] (4,0) .. controls +(0,0.7) and +(-0.4,-0.3) .. (4.8,1.3) ;
\draw [-Stealth] (3.4,1.45) .. controls +(0.6,-0.45) and +(0,0.8) .. (4.25,0) ;
\draw [-Stealth] (3,1.15) .. controls +(0.4,-0.3) and +(0,0.6) .. (3.75,0) ;
\draw [fill=white] (4.05,0.4) circle (0.07) ;
\fill [black] (3.5,0.73) circle (0.07) ;
\end{scope}
\end{tikzpicture}
\caption{Local moves on w-arrows.}\label{Figwarrowmoves}
\end{figure}

As stated in Theorem 4.5 of~\cite{MeiYas}, two arrow presentations represent equivalent diagrams if and only if they are related by these arrow moves. As a result, we can study welded objects through their arrow presentations up to these local moves.

\paragraph{Remark:}
Using the virtual Reidemeister moves $vR1$, $vR2$ and $vR3$, and the fact that w-arrow extremities can pass through virtual crossings, we can assume that the underlying virtual diagram $V$ in an arrow presentation is trivial, i.e. a disjoint union of intervals for a string link and a disjoint union of circles for a link. As a result, when working on arrow presentations with a set number of components, we will only focus on the w-arrows and consider that the underlying virtual diagram is trivial. \\

The use of arrow calculus in this paper revolves around the notion of w-tree, which represent \textquotedblleft stacked commutators" of w-arrows. A w-tree is a unitrivalent tree whose degree one vertices are positioned on the diagram. One of these vertices is the head of the w-tree, and the other ones are tails. There is a cyclic order on the three edges at each trivalent vertex, given by the orientation of the plane in which the diagram is represented. Each edge of the tree can contain a twist. The degree of a w-tree $T$ is defined as the number of tails of $T$. See Figure~\ref{Figwtree} for an example of a w-tree of degree $3$.

\begin{figure}
\centering
\begin{tikzpicture}
\draw [thick] (0,0) -- (0.6,0) ;
\draw [thick] (1,0) -- (1.6,0) ;
\draw [thick] (2,0.5) -- (2.6,0.5) ;
\draw [thick] (1.3,2.3) -- (1.9,2.3) ;
\draw (0.3,0) -- (0.8,0.7) -- (1.3,0) ;
\fill [black] (0.55,0.35) circle (0.06) ;
\fill [black] (0.8,0.7) circle (0.04) ;
\draw (0.8,0.7) -- (1.6,1.4) -- (2.3,0.5) ;
\fill [black] (1.95,0.95) circle (0.06) ;
\fill [black] (1.6,1.4) circle (0.04) ;
\draw [-Stealth] (1.6,1.4) -- (1.6,2.3) ;
\fill [black] (1.6,1.85) circle (0.06) ;
\end{tikzpicture}
\caption{A w-tree of degree $3$.}\label{Figwtree}
\end{figure}
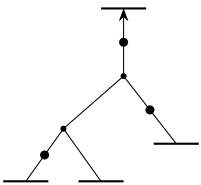

A w-tree represents a collection of w-arrows, obtained by expanding the w-tree iteratively as indicated on Figure~\ref{Figexpwtree}, where $T_1$ and $T_2$ are two unitrivalent trees, and the $T_i$'s obtained after expanding are parallel copies of $T_i$. The surgery along a w-tree is then defined as the surgery along its expansion.

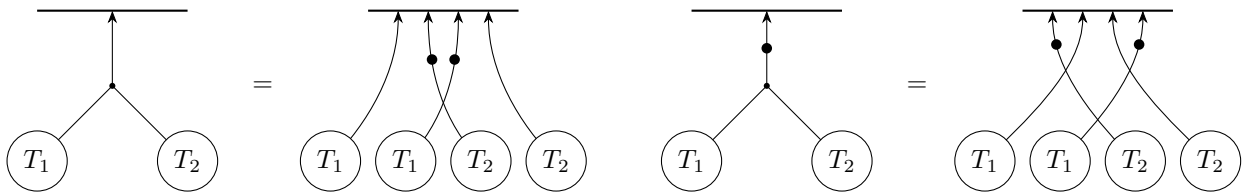
\begin{figure}
\centering
\begin{tikzpicture}
\draw (0,0) -- (1,1) -- (2,0) ;
\fill [black] (1,1) circle (0.04) ;
\draw [-Stealth] (1,1) -- (1,2) ;
\draw [thick] (0,2) -- (2,2) ;
\draw [fill=white] (0,0) circle (0.4) ;
\draw (0,0) node{$T_1$} ;
\draw [fill=white] (2,0) circle (0.4) ;
\draw (2,0) node{$T_2$} ;
\draw (3,1) node{$=$} ;
\draw [thick] (4.4,2) -- (6.4,2) ;
\draw [-Stealth] (3.9,0) .. controls +(0.5,0.5) and +(0,-0.8) .. (4.8,2) ;
\draw [-Stealth] (4.9,0) .. controls +(0.5,0.5) and +(0,-0.8) .. (5.6,2) ;
\draw [-Stealth] (5.9,0) .. controls +(-0.5,0.5) and +(0,-0.8) .. (5.2,2) ;
\draw [-Stealth] (6.9,0) .. controls +(-0.5,0.5) and +(0,-0.8) .. (6,2) ;
\fill [black] (5.25,1.35) circle (0.07) ;
\fill [black] (5.55,1.35) circle (0.07) ;
\draw [fill=white] (3.9,0) circle (0.4) ;
\draw (3.9,0) node{$T_1$} ;
\draw [fill=white] (4.9,0) circle (0.4) ;
\draw (4.9,0) node{$T_1$} ;
\draw [fill=white] (5.9,0) circle (0.4) ;
\draw (5.9,0) node{$T_2$} ;
\draw [fill=white] (6.9,0) circle (0.4) ;
\draw (6.9,0) node{$T_2$} ;

\begin{scope}[xshift=8.7cm]
\draw (0,0) -- (1,1) -- (2,0) ;
\fill [black] (1,1) circle (0.04) ;
\draw [-Stealth] (1,1) -- (1,2) ;
\fill [black] (1,1.5) circle (0.07) ;
\draw [thick] (0,2) -- (2,2) ;
\draw [fill=white] (0,0) circle (0.4) ;
\draw (0,0) node{$T_1$} ;
\draw [fill=white] (2,0) circle (0.4) ;
\draw (2,0) node{$T_2$} ;
\draw (3,1) node{$=$} ;
\draw [thick] (4.4,2) -- (6.4,2) ;
\draw [-Stealth] (3.9,0) .. controls +(0.5,0.5) and +(0,-0.8) .. (5.2,2) ;
\draw [-Stealth] (4.9,0) .. controls +(0.5,0.5) and +(0,-0.8) .. (6,2) ;
\draw [-Stealth] (5.9,0) .. controls +(-0.5,0.5) and +(0,-0.8) .. (4.8,2) ;
\draw [-Stealth] (6.9,0) .. controls +(-0.5,0.5) and +(0,-0.8) .. (5.6,2) ;
\fill [black] (4.85,1.55) circle (0.07) ;
\fill [black] (5.95,1.55) circle (0.07) ;
\draw [fill=white] (3.9,0) circle (0.4) ;
\draw (3.9,0) node{$T_1$} ;
\draw [fill=white] (4.9,0) circle (0.4) ;
\draw (4.9,0) node{$T_1$} ;
\draw [fill=white] (5.9,0) circle (0.4) ;
\draw (5.9,0) node{$T_2$} ;
\draw [fill=white] (6.9,0) circle (0.4) ;
\draw (6.9,0) node{$T_2$} ;
\end{scope}
\end{tikzpicture}
\caption{Expanding a w-tree.}\label{Figexpwtree}
\end{figure}

\begin{de}\label{treepres}
A \emph{tree presentation} of a virtual diagram $D$ is a pair $(V,\T)$, where $V$ is a virtual diagram with no classical crossing and $\T$ is a family of w-trees on $V$ such that the diagram obtained from $V$ by performing the surgeries corresponding to $\T$ is $D$. Two tree presentations are called \emph{equivalent} if they give equivalent virtual diagrams.
\end{de}

The presentation of the group $G_D$ can be adapted to take into account w-trees as follows: the arcs are now delimited not only by undercrossings but also by the heads of the w-trees. Moreover, there is a relation associated to the head of a w-tree, obtained by propagating the generators along the edges from the tails to the head, keeping track of the twists and forming a commutator at each trivalent vertex as illustrated in Figure~\ref{Figpropgen} (note that in~\cite{MeiYas}, $[a,b]$ is defined as $ab^{-1}a^{-1}b$ to match more closely with the expansion of w-trees). \\

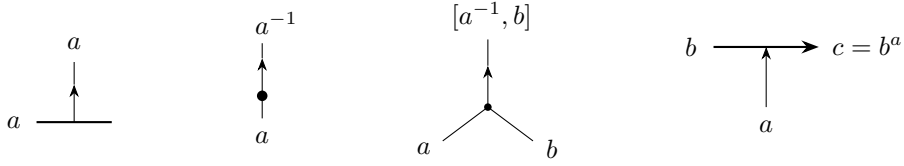
\begin{figure}
\centering
\begin{tikzpicture}
\draw [thick] (-0.5,0) -- (0.5,0) ;
\draw (-0.8,0) node{$a$} ;
\draw [-Stealth] (0,0) -- (0,0.5) ;
\draw (0,0.5) -- (0,0.8) ;
\draw (0,1.05) node{$a$} ;

\begin{scope}[xshift=2.5cm,yshift=0.05cm]
\draw [-Stealth] (0,0) -- (0,0.8) ;
\draw (0,0.8) -- (0,1) ;
\fill [black] (0,0.3) circle (0.07) ;
\draw (0,-0.25) node{$a$} ;
\draw (0.19,1.27) node{$a^{-1}$} ;
\end{scope}

\begin{scope}[xshift=5.5cm,yshift=0.2cm]
\draw (0,0) -- (-0.6,-0.45) ;
\draw (0,0) -- (0.6,-0.45) ;
\draw [-Stealth] (0,0) -- (0,0.55) ;
\draw (0,0.55) -- (0,0.9) ;
\fill [black] (0,0) circle (0.05) ;
\draw (-0.85,-0.55) node{$a$} ;
\draw (0.85,-0.55) node{$b$} ;
\draw (0.05,1.2) node{$[a^{-1},b]$} ;
\end{scope}

\begin{scope}[xshift=9.2cm]
\draw [thick,-Stealth] (-0.7,1) -- (0.7,1) ;
\draw [-Stealth] (0,0.2) -- (0,1) ;
\draw (-1,1) node{$b$} ;
\draw (0,-0.05) node{$a$} ;
\draw (1.35,1.02) node{$c=b^a$} ;
\end{scope}
\end{tikzpicture}
\caption{Propagating generators along w-trees.}\label{Figpropgen}
\end{figure}

As shown in Lemma 5.17 in \cite{MeiYas}, we can exchange the sides of the incoming edges at a trivalent vertex by adding a twist on each of the three edges at this vertex (see Figure~\ref{Figantisym}). This can be used to extrapolate some local moves on w-trees by changing the orientation of certain strands in order to get to a configuration where an established local move can be performed. \\

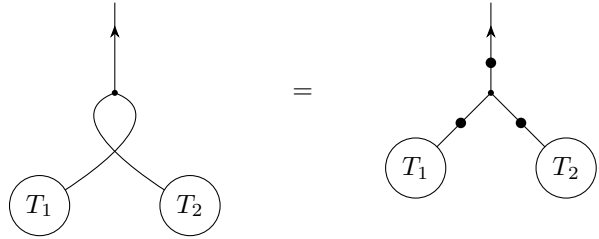
\begin{figure}
\centering
\begin{tikzpicture}
\draw (-1,-1.5) .. controls +(0.7,0.4) and +(0.8,-0.3) .. (0,0) ;
\draw (1,-1.5) .. controls +(-0.7,0.4) and +(-0.8,-0.3) .. (0,0) ;
\draw [-Stealth] (0,0) -- (0,0.9) ;
\draw (0,0.8) -- (0,1.2) ;
\fill [black] (0,0) circle (0.04) ;
\draw [fill=white] (-1,-1.5) circle (0.4) ;
\draw (-1,-1.5) node{$T_1$} ;
\draw [fill=white] (1,-1.5) circle (0.4) ;
\draw (1,-1.5) node{$T_2$} ;

\draw (2.5,0) node{$=$} ;

\begin{scope}[xshift=5cm]
\draw (-1,-1) -- (0,0) -- (1,-1) ;
\draw [-Stealth] (0,0) -- (0,0.9) ;
\draw (0,0.8) -- (0,1.2) ;
\fill [black] (0,0) circle (0.04) ;
\draw [fill=white] (-1,-1) circle (0.4) ;
\draw (-1,-1) node{$T_1$} ;
\draw [fill=white] (1,-1) circle (0.4) ;
\draw (1,-1) node{$T_2$} ;
\fill [black] (0,0.4) circle (0.07) ;
\fill [black] (-0.4,-0.4) circle (0.07) ;
\fill [black] (0.4,-0.4) circle (0.07) ;
\end{scope}
\end{tikzpicture}
\caption{Changing the cyclic order at a trivalent vertex.}\label{Figantisym}
\end{figure}

Local moves on w-arrows induce local moves on w-trees, and such moves are detailed in Section 5.2 of \cite{MeiYas}. We will focus here on the moves that will be needed in the following sections of this paper. First, the inverse move generalizes to w-trees: two parallel copies of the same w-tree which only differ by a twist on the last edge cancel each other out (see Figure~\ref{Figinvwtree}).

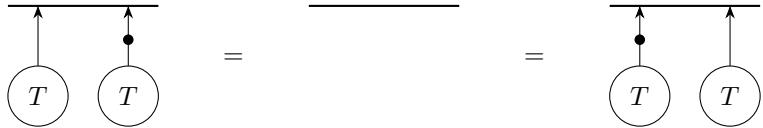
\begin{figure}
\centering
\begin{tikzpicture}
\draw [thick] (0,1.2) -- (2,1.2) ;
\draw [-Stealth] (0.4,0) -- (0.4,1.2) ;
\draw [fill=white] (0.4,0) circle (0.4) ;
\draw (0.4,0) node{$T$} ;
\draw [-Stealth] (1.6,0) -- (1.6,1.2) ;
\draw [fill=white] (1.6,0) circle (0.4) ;
\fill [black] (1.6,0.75) circle (0.07) ;
\draw (1.6,0) node{$T$} ;

\draw (3,0.5) node{$=$} ;
\draw [thick] (4,1.2) -- (6,1.2) ;
\draw (7,0.5) node{$=$} ;

\draw [thick] (8,1.2) -- (10,1.2) ;
\draw [-Stealth] (8.4,0) -- (8.4,1.2) ;
\draw [fill=white] (8.4,0) circle (0.4) ;
\fill [black] (8.4,0.75) circle (0.07) ;
\draw (8.4,0) node{$T$} ;
\draw [-Stealth] (9.6,0) -- (9.6,1.2) ;
\draw [fill=white] (9.6,0) circle (0.4) ;
\draw (9.6,0) node{$T$} ;
\end{tikzpicture}
\caption{Inverse move for w-trees}\label{Figinvwtree}
\end{figure}

The tail exchange move directly generalizes to w-trees: we can permute the positions of two adjacent tails on a strand, whether they belong to the same w-tree or to two distinct ones. The permutation of two adjacent heads can be performed by adding a w-tree whose degree is the sum of the degrees of the two w-trees, as illustrated in Figure~\ref{Figpermheads}. \\

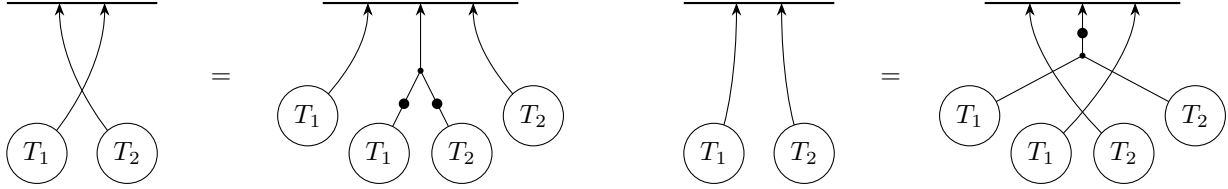
\begin{figure}
\centering
\begin{tikzpicture}
\draw [thick] (0,2) -- (2,2) ;
\draw [-Stealth] (0.4,0) .. controls +(0.4,0.4) and +(0,-0.8) .. (1.3,2) ;
\draw [fill=white] (0.4,0) circle (0.4) ;
\draw (0.4,0) node{$T_1$} ;
\draw [-Stealth] (1.6,0) .. controls +(-0.4,0.4) and +(0,-0.8) .. (0.7,2) ;
\draw [fill=white] (1.6,0) circle (0.4) ;
\draw (1.6,0) node{$T_2$} ;
\draw (2.85,1) node{$=$} ;
\draw [thick] (4.2,2) -- (6.8,2) ;
\draw [-Stealth] (4,0.5) .. controls +(0.2,0.2) and +(0,-0.7) .. (4.8,2) ;
\draw [fill=white] (4,0.5) circle (0.4) ;
\draw (4,0.5) node{$T_1$} ;
\draw (4.95,0) -- (5.5,1.1) -- (6.05,0) ;
\fill [black] (5.5,1.1) circle (0.04) ;
\fill [black] (5.28,0.66) circle (0.07) ;
\fill [black] (5.72,0.66) circle (0.07) ;
\draw [-Stealth] (5.5,1.1) -- (5.5,2) ;
\draw [fill=white] (4.95,0) circle (0.4) ;
\draw (4.95,0) node{$T_1$} ;
\draw [fill=white] (6.05,0) circle (0.4) ;
\draw (6.05,0) node{$T_2$} ;
\draw [-Stealth] (7,0.5) .. controls +(-0.2,0.2) and +(0,-0.7) .. (6.2,2) ;
\draw [fill=white] (7,0.5) circle (0.4) ;
\draw (7,0.5) node{$T_2$} ;

\begin{scope}[xshift=8.8cm]
\draw [thick] (0.2,2) -- (2.2,2) ;
\draw [-Stealth] (0.6,0) .. controls +(0.2,0.4) and +(0,-0.8) .. (0.9,2) ;
\draw [fill=white] (0.6,0) circle (0.4) ;
\draw (0.6,0) node{$T_1$} ;
\draw [-Stealth] (1.8,0) .. controls +(-0.2,0.4) and +(0,-0.8) .. (1.5,2) ;
\draw [fill=white] (1.8,0) circle (0.4) ;
\draw (1.8,0) node{$T_2$} ;
\draw (2.95,1) node{$=$} ;
\draw [thick] (4.2,2) -- (6.8,2) ;
\draw [-Stealth] (4.95,0) .. controls +(0.4,0.3) and +(0,-0.7) .. (6.2,2) ;
\draw [fill=white] (4.95,0) circle (0.4) ;
\draw (4.95,0) node{$T_1$} ;
\draw (4,0.5) -- (5.5,1.3) -- (7,0.5) ;
\fill [black] (5.5,1.3) circle (0.04) ;
\draw [-Stealth] (5.5,1.3) -- (5.5,2) ;
\fill [black] (5.5,1.6) circle (0.07) ;
\draw [fill=white] (4,0.5) circle (0.4) ;
\draw (4,0.5) node{$T_1$} ;
\draw [fill=white] (7,0.5) circle (0.4) ;
\draw (7,0.5) node{$T_2$} ;
\draw [-Stealth] (6.05,0) .. controls +(-0.4,0.3) and +(0,-0.7) .. (4.8,2) ;
\draw [fill=white] (6.05,0) circle (0.4) ;
\draw (6.05,0) node{$T_2$} ;
\end{scope}
\end{tikzpicture}
\caption{Permuting two adjacent heads.}\label{Figpermheads}
\end{figure}

Finally, as detailed in the proof of Lemma 7.10 of \cite{MeiYas}, we can permute the head of a degree $k$ w-tree $T_1$ with a tail of a degree $k'$ w-tree $T_2\neq T_1$ by adding a degree $k+k'$ w-tree, and a collection $C$ of w-trees of degree $>k+k'$ whose heads are positioned next to the head of $T_2$, as indicated in Figure~\ref{Figpermht}.

\begin{figure}
\centering
\begin{tikzpicture}
\draw [thick] (0.1,2) -- (1.6,2) ;
\draw [-Stealth] (0,0.5) .. controls +(0.5,0.2) and +(0,-0.7) .. (1,2) ;
\draw [fill=white] (0,0.5) circle (0.4) ;
\draw (0,0.5) node{$T_1$} ;
\draw [thick] (3,1.3) -- (3,-0.4) ;
\draw (0.5,2) .. controls +(0,-0.7) and +(-0.4,0.3) .. (1.5,0.5) ;
\draw [-Stealth] (1.5,0.5) -- (3,0.5) ;
\draw [fill=white] (1.5,0.5) circle (0.4) ;
\draw (1.5,0.5) node{$T_2$} ;

\draw (4.5,0.5) node{$=$} ;

\begin{scope}[xshift=6.5cm]
\draw [thick] (-0.1,2) -- (1.7,2) ;
\draw [-Stealth] (-0.3,0.5) .. controls +(0.3,0.4) and +(0,-0.7) .. (0.3,2) ;
\draw [fill=white] (-0.3,0.5) circle (0.4) ;
\draw (-0.3,0.5) node{$T_1$} ;
\draw [thick] (3,1.4) -- (3,-2.15) ;
\draw (1.2,2) .. controls +(0,-0.5) and +(-0.3,0.3) .. (1.8,1) ;
\draw [-Stealth] (1.8,1) -- (3,1) ;
\draw [fill=white] (1.8,1) circle (0.4) ;
\draw (1.8,1) node{$T_2$} ;
\draw (0.75,2) -- (0.75,0.4) ;
\draw (0,-0.5) -- (0.75,0.4) -- (1.8,0) ;
\fill [black] (0.75,0.4) circle (0.04) ;
\draw [-Stealth] (1.8,0) -- (3,0) ;
\draw [fill=white] (0,-0.5) circle (0.4) ;
\draw (0,-0.5) node{$T_1$} ;
\draw [fill=white] (1.8,0) circle (0.4) ;
\draw (1.8,0) node{$T_2$} ;
\draw [-Stealth] (1.8,-0.8) -- (3,-0.8) ;
\draw (2.55,-1.12) node{$\vdots$} ;
\draw [-Stealth] (1.8,-1.6) -- (3,-1.6) ;
\draw [fill=white] (1.8,-1.2) ellipse (0.4 and 0.6) ;
\draw (1.8,-1.2) node{$C$} ;
\end{scope}
\end{tikzpicture}
\caption{Permuting a head and a tail from two different w-trees.}\label{Figpermht}
\end{figure}
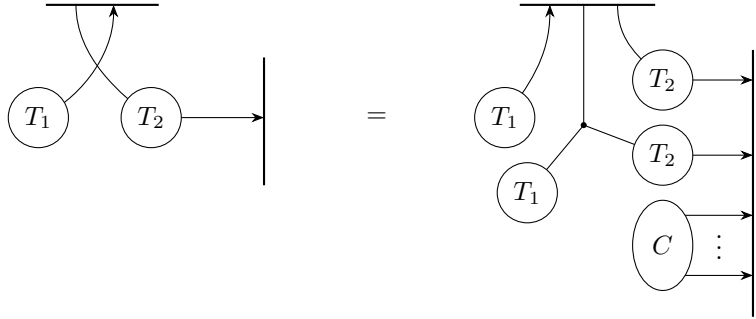

\paragraph{Remark:}
We will also need to be able to permute the head of a w-tree with one of its tails. We will see in Section~\ref{Sec2} that, up to the addition of w-trees of higher degree, it can be done using the preceding moves up to $w_q$--equivalence (see Definition~\ref{wqequiv}) and welded concordance, which is defined in Section~\ref{Sec13}. \\

One important aspect of the moves described above is that the heads of the additional w-trees of higher degree are positioned next to some pre-existing heads. \\

The notion of $w_q$--equivalence is defined in Section 7.1 of \cite{MeiYas} as follows:

\begin{de}\label{wqequiv}
Two virtual diagrams are said to be \emph{$w_q$--equivalent} if one can be obtained from the other by a finite sequence of welded Reidemeister moves and surgeries along w-trees of degree $\geq q$.
\end{de}

An important step in the classification of welded string links (resp. welded links) will be showing that any virtual diagram of welded string link (resp. welded link) is equivalent up to $w_q$--equivalence and welded concordance to a diagram admitting an ascending tree presentation (resp. is $w_q$--equivalent to a diagram admitting a sorted tree presentation), defined as follows:

\begin{de}\label{defascsort}
A tree presentation of a welded string link is called \emph{ascending} if on each component, the heads of the w-trees are positioned above the tails (the strands being oriented upwards). A tree presentation of a welded link is called \emph{sorted} if on each component, the heads and the tails of the w-trees are positioned in two disjoint intervals.
\end{de}

The corresponding notion for diagrams of tangles was introduced by D. Bar-Natan, Z. Dansco and R. van der Veen in~\cite{OtUT}, where it is called \emph{over-then-under tangles}, as all the overcrossings appear before the undercrossings on each strand. The closely related notions of ascending and sorted Gauss diagrams were introduced in~\cite{RTwSL} and \cite{CRPSL} respectively.

\subsection{Welded concordance}\label{Sec13}

In this section, we introduce the notion of welded concordance between two welded links or two welded string links. This notion was defined on virtual objects by J. Carter and S. Kamada in \cite{SEKSVKC}, and adapted to welded objects by Gaudreau in \cite{Gau} and by H. Boden and M. Chrisman in \cite{BodChris}. \\

A \emph{cobordism of welded links} is a finite sequence of virtual diagrams of welded links (with potentially differing number of components), where two consecutive diagrams differ by one of the following moves:
\begin{itemize}
\item a welded Reidemeister move;
\item a \emph{birth} or a \emph{death}, i.e. the creation or deletion of a closed component with no crossings;
\item a \emph{saddle point}, i.e. the fusion of two components or separation of a component in two as illustrated below.
\end{itemize}
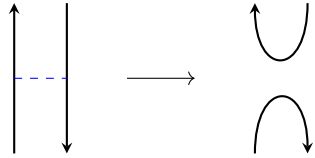
\begin{center}
\begin{tikzpicture}
\draw [thick,-stealth] (0,0) -- (0,2) ;
\draw [thick,-stealth] (0.7,2) -- (0.7,0) ;
\draw [blue,dashed] (0,1) -- (0.7,1) ;
\draw [->] (1.5,1) -- (2.4,1) ;
\draw [thick,-stealth] (3.2,0) .. controls +(0,1) and +(0,1) .. (3.9,0) ;
\draw [thick,-stealth] (3.9,2) .. controls +(0,-1) and +(0,-1) .. (3.2,2) ;
\end{tikzpicture}
\end{center}

In order to define cobordisms of welded string links, we need to introduce the notion of mixed virtual diagrams, which is defined in the same way as a virtual diagram of a welded string link, but with the addition of some (possibly none) oriented closed components. A \emph{cobordism of $n$--component welded string links} is a finite sequence of mixed virtual diagrams, each containing $n$ interval components, where two consecutive diagrams differ by one of the moves listed above. \\

Given a cobordism of welded links or welded string links, we can join the consecutive diagrams to form an immersed surface, capped by disks at births and deaths, and with saddles at saddle points. The surface $S$ contains a certain number of connected components $S_i$, with Euler characteristic $\chi(S_i)$.

\begin{de}\label{wldconcdef}
A \emph{welded concordance} is a cobordism of welded links (resp. welded string links) between two virtual diagrams of $n$--component welded links (resp. welded string links), satisfying the following conditions:
\begin{itemize}
\item the underlying surface $S$ has $n$ connected components, connecting the $i^{th}$ component of the first diagram to the $i^{th}$ component of the last diagram;
\item each connected component of $S$ is homeomorphic to an annulus $S^1\times I$ (resp. a disk $I\times I$).
\end{itemize}
\end{de}

The definition of a saddle point guaranties that the underlying surface of a cobordism of welded links (resp. welded string links) is orientable. Since the boundary of each connected component $S_i$ is a disjoint union of two circles (resp. a circle), its homeomorphic class is determined by its Euler characteristic, which must be $0$ (resp. $1$) for it to be an annulus (resp. a disk). We can compute $\chi(S_i)$ by counting the number of births, deaths and saddle points involving the connected component $S_i$. More precisely, for a welded link (resp. welded string link), the Euler characteristic of $S_i$ starts at $0$ (resp. $1$), increases by $1$ for each birth and each death, and decreases by $1$ for each saddle point. As a result, $S_i$ is homeomorphic to an annulus (resp. a disk) if and only if it involves as many births and deaths as saddle points.

\begin{de}
Two virtual diagrams are said to be \emph{$w_q$--concordant} if one can be obtained from the other by a finite sequence of $w_q$--equivalences and welded concordance.
\end{de}

\section{Classification of welded string links}\label{Sec2}

In this section, we introduce the Milnor invariants of welded string links, which give invariants of $w_q$--concordance. The main goal is to prove that the Milnor invariants of length $\leq q$ are complete invariants of $w_q$--concordance. In order to do that, we will show that any tree presentation of a welded string link is $w_q$--concordant to an ascending presentation, as the classification result on ascending presentations follows from the fact that longitudes can be read directly on diagrams. This will in turn give a diagrammatical characterization of classical string links with equal Milnor invariants. \\

Since this section is dedicated to welded string links and will not involve welded links, we refer to virtual diagrams of $n$--component welded string links as simply virtual diagrams. \\

%\subsection{Chen maps and Milnor invariants}\label{Sec21}

At the end of Section 1.2.1, we defined the peripheral system of a welded string link $L$ represented by a virtual diagram $D$. We will now use it to obtain the Milnor invariants of $L$. More precisely, we will consider the \emph{$q$--nilpotent peripheral system} of $L$, which consists of the triplet $(N_qG_D,\ov{m},\ov{\ell})$, where $\ov{m}=(\ov{m}_1,\ldots,\ov{m}_n)\in(N_qG_D)^n$ is the image of the meridians in the $q$--niloptent quotient of $G_D$, and $\ov{\ell}=(\ov{\ell}_1,\ldots,\ov{\ell}_n)\in(N_qG_D)^n$ is the image of the longitudes. \\

It is known that the $q$--nilpotent quotient $N_qG_D$ is generated by the $\ov{m}_i$'s, so each longitude $\ov{\ell}_j$ can be expresed as a word on the $\ov{m}_i$'s. The power series $E_q(\ov{\ell}_j)$ is then well defined modulo monomials of degree $\geq q$ (see~\cite{CAMIWL} for more details). For a sequence of indices $I=(i_1,\ldots,i_k,j)\in[\![1,n]\!]^{k+1}$ with $k<q$, we denote by $\mu_I(D)\in\ZZ$ the coefficient of the monomial $X_{i_1}\cdots X_{i_k}$ in $E_q(\ov{\ell}_j)$ the \emph{Milnor invariant of index $I$ of $D$}. The length of a Milnor invariant $\mu_I(D)$ is defined as the length of the sequence of indices $I$. \\

%[historical references] \\

Since the effect of a surgery along a w-tree of degree $k$ is to add a commutator of weight $k$ in one of the longitudes $\ell_j$, it does not modify $\ov{\ell}_j\in N_qG_D$ if $k\geq q$, and as such does not modify the Milnor invariants of length $\leq q$. As a result, we have the following, as noted in Remark 7.6 of~\cite{MeiYas}:

\begin{prop}\label{invwqconc}
The Milnor invariants of length $\leq q$ are invariants of $w_q$--equivalence on welded string links.
\end{prop}

%\subsection{Classification of welded string links up to $w_q$--concordance}\label{Sec22}

%In this section, we show that the Milnor invariants are invariants of welded concordance, then that every tree presentation of a welded string link is $w_q$--concordant to an ascending presentation. Using the fact that ascending presentations with equal Milnor invariants of length $\leq q$ are $w_q$--equivalent, we obtain the classification of welded string links up to $w_q$--concordance.

\begin{prop}\label{invconc}{\cite[Cor. 5.6]{MCIKSB}}
The Milnor invariants are welded concordance invariants.
\end{prop}

This can be established using the terminology of cut-diagrams introduced in~\cite{MCIKSB}, which generalizes Gauss diagrams in higher dimensions. It can easily be checked that our definition of welded concordance (Definition~\ref{wldconcdef}) between two welded string links represented by virtual diagrams implies the existence of a cut-diagram of a surface between the two, where the lines of double points correspond to the preimages of the undercrossings. As a result, the two welded string links are cut-concordant (see Definition 5.1 in~\cite{MCIKSB}), and by Corollary 5.6 of~\cite{MCIKSB} cut-concordance preserves the Milnor invariants, which proves Proposition~\ref{invconc}.

\paragraph{Remark:}
A topological proof of Proposition~\ref{invconc} was given in \cite{MCICHR} by M. Chrisman using the Tube map, which associates a knotted surface in $\RR^4$ to a virtual (or welded) object. This knotted surface has the same Milnor invariants as the diagrammatical object, and the invariance under topological concordance of surfaces follows from a theorem of Stallings. \\

We will now prove that the Milnor invariants of length $\leq q$ are complete invariants of $w_q$--concordance on welded string links. The main step in the proof is showing that any tree presentation is $w_q$--concordant to an ascending presentation, for which the classification result is relatively straightforward. \\

Using the moves described in Section~\ref{Sec12}, we have the following result:

\begin{prop}\label{permttht}
Let $T_1$ and $T_2$ be two different w-trees of respective degree $k$ and $k'$ in a tree presentation $(V,\T)$. Then $(V,\T)$ is equivalent to a tree presentation $(V,\T')$ which only differs from $(V,\T)$ by the permutation of two adjacent ends of $T_1$ and $T_2$ and the addition of w-trees of degree $\geq k+k'$ whose heads are positioned next to the heads of w-trees in $(V,\T)$. In particular, the degrees of the additional w-trees are greater than the degrees of $T_1$ and $T_2$.
\end{prop}

Since the tail exchange move generalizes to w-trees, we can permute two adjacent tails without adding any new w-tree, whether they come from the same w-tree or not. We now show how to perform the last move needed to get to an ascending presentation.

\begin{prop}\label{permht}
Let $q\geq 1$, and let $T$ be a w-tree of degree $k$ in a tree presentation $(V,\T)$, such that $T$ has a tail adjacent to its head. Then $(V,\T)$ is $w_q$--concordant to a tree presentation $(V,\T')$ which only differs from $(V,\T)$ by the permutation of the head of $T$ with its tail and the addition of w-trees of degree $>k$. Moreover, the heads of the additional w-trees are positioned in a small neighborhood of the head of $T$.
\end{prop}

\begin{proof}
If $k\geq q$, we can simply delete $T$ by $w_q$--equivalence and recreate it with the positions of its head and adjacent tail swapped on the strand. We illustrate the proof for $k<q$ on Figure~\ref{FigpermhtwqC}, where we do not represent the additional w-trees of degree $>k$ created by the moves described in Proposition~\ref{permttht}. The heads of these w-trees all end up on the portion of strand represented in the figure. \\

To go from picture (a) to picture (b), we use a birth to create an empty closed component. To go from (b) to (c), we use the inverse move to introduce two copies of $T$ parallel to each other which only differ by a twist on the last edge, but whose head and adjacent tail are positioned on the closed component. To go from (c) to (d), we use Proposition~\ref{permttht} to permute the head of the upper copy of $T$ with the tail of the lower one. This creates w-trees of degree $\geq 2k$ whose heads are positioned next to the heads of these two copies of $T$. To go from (d) to (e), we slide the tail of the upper copy of $T$ along the closed component before performing a saddle move on picture (f). Picture (g) depicts the same configuration as picture (f) after straightening the portion of strand. To go from (g) to (h), we use Proposition~\ref{permttht} to bring the tail of the upper copy of $T$ up to its head, and the head of the middle copy of $T$ next to the head of the lower copy of $T$. There might be heads and tails of higher degree w-trees in the way, but since we can delete g-trees of degree $\geq q$ by $w_q$--equivalence, we can get to the configuration depicted in picture (h) in a finite number of steps. Finally, to go from (h) to (i) we use the inverse move to cancel out the two lower copies of $T$. We are left with a copy of $T$ whose head and adjacent tail have been permuted, and a collection of additional w-trees of degree $>k$ whose heads are positioned on the portion of strand displayed in picture (i).
\end{proof}

\begin{figure}
\centering
\begin{tikzpicture}
\begin{scope}[xshift=-5.5cm]
\draw [thick,-Stealth] (0,0) -- (0,3) ;
\draw [-Stealth] (-0.8,1.3) -- (0,1.3) ;
\draw (-0.8,1.7) -- (0,1.7) ;
\draw [fill=white] (-0.8,1.5) ellipse (0.3 and 0.35) ;
\draw (-0.8,1.5) node{$T$} ;
\draw (-0.3,-0.5) node{(a)} ;
\end{scope}

\begin{scope}[xshift=-2.8cm]
\draw [thick,-Stealth] (0,0) -- (0,3) ;
\draw [-Stealth] (-0.8,1.3) -- (0,1.3) ;
\draw (-0.8,1.7) -- (0,1.7) ;
\draw [fill=white] (-0.8,1.5) ellipse (0.3 and 0.35) ;
\draw (-0.8,1.5) node{$T$} ;
\draw [thick,-Stealth] (0.5,1.5) -- (0.5,1.4) ;
\draw [thick] (1.4,1.5) ellipse (0.9 and 1.2) ;
\draw (0.5,-0.5) node{(b)} ;
\end{scope}

\begin{scope}[xshift=2cm]
\draw [thick,-Stealth] (0,0) -- (0,3) ;
\draw [-Stealth] (-0.8,1.3) -- (0,1.3) ;
\draw (-0.8,1.7) -- (0,1.7) ;
\draw [fill=white] (-0.8,1.5) ellipse (0.3 and 0.35) ;
\draw (-0.8,1.5) node{$T$} ;
\draw [thick,-Stealth] (0.5,1.5) -- (0.5,1.4) ;
\draw [thick] (1.4,1.5) ellipse (0.9 and 1.2) ;
\draw [-Stealth] (1.4,1.7) -- (2.28,1.3) ;
\draw (1.4,2.1) -- (2.18,2.1) ;
\draw [fill=white] (1.4,1.9) ellipse (0.3 and 0.35) ;
\draw (1.4,1.9) node{$T$} ;
\draw [-Stealth] (1.4,0.9) -- (2.18,0.9) ;
\draw (1.4,1.3) -- (2.28,1.7) ;
\draw [fill=white] (1.4,1.1) ellipse (0.3 and 0.35) ;
\draw (1.4,1.1) node{$T$} ;
\fill [black] (1.85,0.9) circle (0.06) ;
\draw (0.5,-0.5) node{(c)} ;
\end{scope}

\begin{scope}[xshift=6.8cm]
\draw [thick,-Stealth] (0,0) -- (0,3) ;
\draw [-Stealth] (-0.8,1.3) -- (0,1.3) ;
\draw (-0.8,1.7) -- (0,1.7) ;
\draw [fill=white] (-0.8,1.5) ellipse (0.3 and 0.35) ;
\draw (-0.8,1.5) node{$T$} ;
\draw [thick,-Stealth] (0.5,1.5) -- (0.5,1.4) ;
\draw [thick] (1.4,1.5) ellipse (0.9 and 1.2) ;
\draw [-Stealth] (1.4,1.8) -- (2.27,1.8) ;
\draw (1.4,2.2) -- (2.12,2.2) ;
\draw [fill=white] (1.4,2) ellipse (0.3 and 0.35) ;
\draw (1.4,2) node{$T$} ;
\draw [-Stealth] (1.4,0.8) -- (2.12,0.8) ;
\draw (1.4,1.2) -- (2.27,1.2) ;
\draw [fill=white] (1.4,1) ellipse (0.3 and 0.35) ;
\draw (1.4,1) node{$T$} ;
\fill [black] (1.83,0.8) circle (0.06) ;
\draw (0.5,-0.5) node{(d)} ;
\end{scope}

\begin{scope}[xshift=-5.5cm,yshift=-4.5cm]
\draw [blue,dashed,thick] (0,1.5) -- (0.5,1.5) ;
\draw [thick,-Stealth] (0,0) -- (0,3) ;
\draw [-Stealth] (-0.8,0.6) -- (0,0.6) ;
\draw (-0.8,1) -- (0,1) ;
\draw [fill=white] (-0.8,0.8) ellipse (0.3 and 0.35) ;
\draw (-0.8,0.8) node{$T$} ;
\draw [thick,-Stealth] (0.5,1.5) -- (0.5,1.4) ;
\draw [thick] (1.4,1.5) ellipse (0.9 and 1.2) ;
\draw [-Stealth] (1.3,1.8) -- (2.27,1.8) ;
\draw (1.3,2.2) .. controls +(0.7,0.3) and +(0.2,0.8) .. (0.8,2) ;
\draw (0.8,2) .. controls +(-0.1,-0.35) and +(0.3,0.1) .. (0.55,1.2) ;
\draw [fill=white] (1.3,2) ellipse (0.3 and 0.35) ;
\draw (1.3,2) node{$T$} ;
\draw [-Stealth] (1.4,0.8) -- (2.12,0.8) ;
\draw (1.4,1.2) -- (2.27,1.2) ;
\draw [fill=white] (1.4,1) ellipse (0.3 and 0.35) ;
\draw (1.4,1) node{$T$} ;
\fill [black] (1.83,0.8) circle (0.06) ;
\draw (0.5,-0.5) node{(e)} ;
\end{scope}

\begin{scope}[xshift=-0.75cm,yshift=-4.5cm]
\draw [thick] (0,0) -- (0,1) ;
\draw [thick] (0,1) .. controls +(0,0.5) and +(-0.1,0.3) .. (0.3,1) ;
\draw [thick] (0.3,1) .. controls +(0.4,-1.2) and +(0,-1.5) .. (2,1.5) ;
\draw [thick] (2,1.5) .. controls +(0,1.5) and +(0.4,1.2) .. (0.3,2) ;
\draw [thick] (0.3,2) .. controls +(-0.1,-0.3) and +(0,-0.5) .. (0,2) ;
\draw [thick,-Stealth] (0,2) -- (0,3) ;
\draw [-Stealth] (-0.8,0.6) -- (0,0.6) ;
\draw (-0.8,1) -- (0,1) ;
\draw [fill=white] (-0.8,0.8) ellipse (0.3 and 0.35) ;
\draw (-0.8,0.8) node{$T$} ;
\draw [-Stealth] (1,1.8) -- (1.97,1.8) ;
\draw (1,2.2) .. controls +(0.7,0.3) and +(0.2,0.8) .. (0.5,2) ;
\draw (0.5,2) .. controls +(-0.1,-0.35) and +(0.3,0.1) .. (0.27,1.1) ;
\draw [fill=white] (1,2) ellipse (0.3 and 0.35) ;
\draw (1,2) node{$T$} ;
\draw [-Stealth] (1.2,0.8) -- (1.88,0.8) ;
\draw (1.2,1.2) -- (1.98,1.2) ;
\draw [fill=white] (1.2,1) ellipse (0.3 and 0.35) ;
\draw (1.2,1) node{$T$} ;
\fill [black] (1.61,0.8) circle (0.06) ;
\draw (0.5,-0.5) node{(f)} ;
\end{scope}

\begin{scope}[xshift=4cm,yshift=-4.5cm]
\draw [thick,-Stealth] (0,0) -- (0,3) ;
\draw [-Stealth] (-0.8,0.4) -- (0,0.4) ;
\draw (-0.8,0.8) -- (0,0.8) ;
\draw [fill=white] (-0.8,0.6) ellipse (0.3 and 0.35) ;
\draw (-0.8,0.6) node{$T$} ;
\draw [-Stealth] (-1,1.3) -- (0,1.3) ;
\draw (-1,1.7) -- (0,1.7) ;
\draw [fill=white] (-1,1.5) ellipse (0.3 and 0.35) ;
\draw (-1,1.5) node{$T$} ;
\fill [black] (-0.27,1.3) circle (0.06) ;
\draw [-Stealth] (-1,2.2) -- (0,2.2) ;
\draw (-1,2.7) .. controls +(1,0) and +(-1,0) .. (0,1.05) ;
\draw [fill=white] (-1,2.4) ellipse (0.3 and 0.35) ;
\draw (-1,2.4) node{$T$} ;
\draw (-0.3,-0.5) node{(g)} ;
\end{scope}

\begin{scope}[xshift=6.75cm,yshift=-4.5cm]
\draw [thick,-Stealth] (0,0) -- (0,3) ;
\draw [-Stealth] (-0.8,0.4) -- (0,0.4) ;
\draw (-0.8,0.8) -- (0,1.3) ;
\draw [fill=white] (-0.8,0.6) ellipse (0.3 and 0.35) ;
\draw (-0.8,0.6) node{$T$} ;
\draw [-Stealth] (-0.8,1.3) -- (0,0.8) ;
\draw (-0.8,1.7) -- (0,1.7) ;
\draw [fill=white] (-0.8,1.5) ellipse (0.3 and 0.35) ;
\draw (-0.8,1.5) node{$T$} ;
\fill [black] (-0.24,0.95) circle (0.06) ;
\draw [-Stealth] (-0.8,2.2) .. controls +(0.6,0) and +(-0.6,0) .. (0,2.6) ;
\draw (-0.8,2.6) .. controls +(0.6,0) and +(-0.6,0) .. (0,2.2) ;
\draw [fill=white] (-1,2.4) ellipse (0.3 and 0.35) ;
\draw (-1,2.4) node{$T$} ;
\draw (-0.3,-0.5) node{(h)} ;
\end{scope}

\begin{scope}[xshift=9.5cm,yshift=-4.5cm]
\draw [thick,-Stealth] (0,0) -- (0,3) ;
\draw [-Stealth] (-0.8,1.3) .. controls +(0.6,0) and +(-0.6,0) .. (0,1.7) ;
\draw (-0.8,1.7) .. controls +(0.6,0) and +(-0.6,0) .. (0,1.3) ;
\draw [fill=white] (-1,1.5) ellipse (0.3 and 0.35) ;
\draw (-1,1.5) node{$T$} ;
\draw (-0.3,-0.5) node{(i)} ;
\end{scope}
\end{tikzpicture}
\caption{Permuting the head and a tail of the same w-tree.}
{\small The additional w-trees of higher degree are not represented.}\label{FigpermhtwqC}
\end{figure}

\begin{cor}\label{equivascdiag}
Let $(V,\T)$ be a tree presentation of a welded string link. For any $q\geq 1$, $(V,\T)$ is $w_q$--concordant to an ascending presentation.
\end{cor}

\begin{proof}
First, note that by $w_q$--equivalence all w-trees of degree $\geq q$ can be deleted, and this will be done implicitly in all the steps of this proof. Using the moves $vR1$, $vR2$ and $vR3$ on the underlying virtual diagram $V$, we can assume that it is made up of vertical strands without any virtual crossing. We start by sliding the extremities of the w-trees of $\T$ upward (without permuting any) in order to get to a configuration where all the heads are above the midway point on each strand. The goal is then to get to a configuration where all the tails of w-trees are below the midway point, since such a presentation would be ascending. \\

Let $t=(t_1,\ldots,t_{q-1})\in\NNo^{q-1}$ be the $(q-1)$--tuple indicating the number $t_i$ of tails belonging to w-trees of degree $i$ which are above the midway point. We will proceed by induction on $c$, the set $\NNo^{q-1}$ being given the lexicographic order (which is a well-order, so the principle of induction applies). By Propositions~\ref{permttht} and~\ref{permht}, we can permute two adjacent ends of w-trees which are above the midway point by adding w-trees of higher degree, whose heads are positioned above the midway point as they appear next to pre-existing heads, which are all above it. If $t=(0,\ldots,0)$, we are done. If $t>(0,\ldots,0)$, let $T$ be a w-tree with a tail $\tau$ above the midway point, and whose degree is minimal among such w-trees. Let $h=(h_1,\ldots,h_{q-1})\in\NNo^{q-1}$ the $(q-1)$--tuple indicating the number $h_i$ of heads belonging to w-trees of degree $i$  which are positionned between the tail $\tau$ and the midway point. We will bring the tail $\tau$ under the midway point by induction on $h$. \\

If $h=(0,\ldots,0)$, there are no heads between $\tau$ and the midway point, so we can slide $\tau$ under it. If $h>(0,\ldots,0)$, we slide $\tau$ down to the first head below it. Let $j$ be the degree of the w-tree of this head. We can permute $\tau$ with this head by adding w-trees of degree greater than $\deg(T)$ and $j$, whose heads are above the midway point. We have replaced $h=(h_1,\ldots,h_j,h_{j+1},\ldots,h_{q-1})$ by $h'=(h_1,\ldots,h_j-1,h_{j+1}',\ldots,h_{q-1}')<h$. By the induction hypothesis, we can then bring $\tau$ below the midway point. \\

By doing this, we created new w-trees of degree $>i:=\deg(T)$ and whose heads are above the midway point, so we replaced $t=(0,\ldots,0,t_i,t_{i+1},\ldots,t_{q-1})$ by $t'=(0,\ldots,0,t_i-1,t_{i+1}',\ldots,t_{q-1}')<t$. This concludes the proof by induction on $t$.
\end{proof}

\begin{lem}\label{classascdiag}
Two ascending presentations which have equal Milnor invariants of length $\leq q$ are $w_q$--equivalent.
\end{lem}

\begin{proof}
Let $(V,\T)$ be an ascending presentation. After expanding the w-trees, we can assume that we have an arrow presentation $(V,\A)$. The conjugating elements associated to the w-arrows are then meridians $m_j^{\pm}$, i.e. the generators associated to the bottom arcs of the virtual diagram. As a result, the longitudes can be directly read as words on the meridians, and there is a one-to-one correspondance between words $\ell_1,\ldots,\ell_n$ in the $m_j^{\pm 1}$'s and ascending arrow presentations, each $m_j^{\pm 1}$ in $\ell_i$ corresponding to a appropriately twisted w-arrow from the $j^{th}$ component to the $i^{th}$ one. Using isolated arrow moves at the bottom of $i^{th}$ component, we can replace $\ell_i$ by $m_i^k\ell_i$, so we can assume that $a_i^{\ast}(\ell_i)=0$ for $1\leq i\leq n$, so that the $\ell_i$'s are the longitudes of the diagram. We will show that an ascending tree presentation $(V,\T')$ with Milnor invariants of length $\leq q$ equal to those of $(V,\A)$ can be modified using $w_q$--equivalence into an ascending arrow presentation $(V,\A'')$ which has the same $\ell_i$'s as $(V,\A)$, and thus represent equivalent virtual diagrams. \\

Let us write the longitudes $\ell_i'$ of $(V,\T')$ as words on the generators $m_j$'s instead of the $m_j'$'s. By the equality of the Milnor invariants of legnth $\leq q$, we have $E(\ell_i)\equiv E(\ell_i')\!\mod\X^q$. By Proposition~\ref{injMagexp}, we have $\ell_i\equiv\ell_i'\!\mod\G_qF_n$. By Proposition~\ref{gencom}, $\ell_i$ and $\ell_i'$ differ by a finite sequence of insertion/deletion of $m_jm_j^{-1}$ and $m_j^{-1}m_j$, which can be performed on arrow presentations by inverse moves, and insertion/deletion of commutators of weight $\geq q$ on the $m_j^{\pm 1}$'s, which can be performed on tree presentations by insertion/deletion of suitable w-trees of degree $\geq q$ by $w_q$--equivalence. As a result, $(V,\T')$ is $w_q$--equivalent to an ascending tree presentation $(V,\T'')$ such that $\ell_i''=\ell_i$ as words on the $m_j^{\pm 1}$'s. After expanding the w-trees to obtain an arrow presentation $(V,\A'')$, we thus have $(V,\A'')=(V,\A)$, so $(V,\A'')$ and $(V,\T')$ are $w_q$--equivalent.
\end{proof}

Using Propositions~\ref{invconc} and~\ref{invwqconc}, Corollary~\ref{equivascdiag} and Lemma~\ref{classascdiag}, we obtain the main result of the section:

\begin{theo}\label{classwldstrlk}
The Milnor invariants of length $\leq q$ are complete invariants of $w_q$--concordance on welded string links.
\end{theo}

As a result, we obtain the following diagrammatical characterization on classical string links:

\begin{cor}\label{classclsstrlk}
Two classical string links have equal Milnor invariants of length $\leq q$ if and only if they are $w_q$--concordant as welded string links.
\end{cor}

We also have the following corollary on the finite type concordance invariants of welded string links:

\begin{cor}
For two welded string links $L$ and $L'$, the following statements are equivalent:
\begin{enumerate}
\item $L$ and $L'$ are $w_q$--concordant,
\item $L$ and $L'$ have the same Milnor invariants of length $\leq q$,
\item $L$ and $L'$ the same finite type concordance invariants of degree $<q$.
\end{enumerate}
\end{cor}

\begin{proof}
We know that $1\Rightarrow 3$ by Proposition 7.5 of~\cite{MeiYas}, and $3\Rightarrow 2$ was proved in the appendix of~\cite{Cas21} using arguments similar to those of D. Bar-Natan in the classical case. Finally, Theorem~\ref{classwldstrlk} above gives $2\Leftrightarrow 1$.
\end{proof}

This shows that on welded string links, the Milnor invariants are universal among $\ZZ$-valued finite type concordance invariants. This is to be compared to the result of N. Habegger and G. Masbaum (Corollary 6.4 in~\cite{TKIMI}) giving the universality of the Milnor invariants among the rational finite type concordance invariants on classical string links. \\

Since the Milnor invariants of any length are welded concordance invariants, and the restriction to invariants of length $\leq q$ for the classification of $w_q$--concordance comes from the addition of $w_q$--equivalence, a natural question is to ask whether the Milnor invariants are complete invariants of welded concordance on welded string links, or if other invariants are required. This is still an open question.

\section{Classification of welded links}\label{Sec3}

The goal of this section is to prove that welded links are classified by their $q$--nilpotent peripheral systems up to $w_q$--concordance. We will use the fact that any tree presentation of a welded link is $w_q$--equivalent to a sorted presentation, and that two sorted presentations with equivalent $q$--nilpotent peripheral systems (see Definition~\ref{defequivqnps}) are $w_q$--concordant. We will then use this classification to characterize welded links, and in turn classical links, with vanishing Milnor invariants of length $\leq q$. \\

Since this section is dedicated to welded links, we refer to virtual diagrams of $n$--component welded links as simply virtual diagrams.

\subsection{The $q$--nilpotent peripheral system of a welded link}\label{Sec31}

We defined the peripheral system of a virtual diagram $D$ at the end of Section~\ref{Sec121} as the group $G_D$ together with meridians $m_i$ and longitudes $\ell_i$. Similarly to welded string links, we consider the \emph{$q$--nilpotent peripheral system of $D$} $(N_qG_D,\ov{m},\ov{\ell})$. In order to remove the indeterminacy coming from the choice of the meridians, we introduce the following equivalence relation on $q$--nilpotent peripheral systems:

\begin{de}\label{defequivqnps}
A $q$--nilpotent peripheral system $(N_qG_D,\ov{m},\ov{\ell})$ of a virtual diagram $D$ is said to be \emph{equivalent} to a $q$--nilpotent peripheral system $(N_qG_{D'},\ov{m}',\ov{\ell}')$ of a virtual diagram $D'$ if there exists an isomorphism $\varphi:N_qG_D\to N_qG_{D'}$ and elements $g_1,\ldots,g_n\in N_qG_D$ such that $\varphi\big(\ov{m}_i^{g_i}\big)=\ov{m}_i'$ and $\varphi\big(\ov{\ell}_i^{g_i}\big)=\ov{\ell}_i'$ for $1\leq i\leq n$.
\end{de}

\begin{prop}\label{equivqnps}
Two $q$--nilpotent peripheral systems of a virtual diagram $D$ are equivalent.
\end{prop}

\begin{proof}
Let $(N_qG_D,\ov{m},\ov{\ell})$ and $(N_qG_D,\ov{m}',\ov{\ell}')$ be two $q$--nilpotent peripheral systems of a virtual diagram $D$. Let $p_i$ (resp. $p_i'$) denote the point on $D_i$ associated to $m_i$ (resp. $m_i'$), let $J_i$ denote the subinterval $]p_i,p_i'[$ of $D_i$, and let $g_i:=w_{J_i}\in N_qG_D$. Due to the Wirtinger relations associated to the undercrossings on $J_i$, we have $\ov{m}_i'=\ov{m}_i^{g_i}$.

\begin{center}
\begin{tikzpicture}
\draw (-2.4,-0.03) node{$D_i$} ;
\draw [thick,dashed] (-2,0) -- (-1.5,0) ;
\draw [thick] (-1.5,0) -- (-1.25,0) ;
\draw [thick] (1.25,0) -- (1.5,0) ;
\draw [thick,dashed] (1.5,0) -- (2,0) ;
\draw [thick,red] (-1.25,0) -- (1.25,0) ;
\draw [thick,red,-stealth] (-0.1,0) -- (0.1,0) ;
\draw [thick] (-1.25,-0.1) -- (-1.25,0.1) ;
\draw [thick] (1.25,-0.1) -- (1.25,0.1) ;
\draw (-1.25,0.35) node{$p_i$} ;
\draw (-1.25,-0.4) node{$\ov{m}_i$} ;
\draw (1.25,0.4) node{$p_i'$} ;
\draw (1.7,-0.4) node{$\ov{m}_i'=\ov{m}_i^{g_i}$} ;
\draw [red] (0,-0.45) node{$J_i$} ;
\fill [white] (-0.65,0) circle (0.1) ;
\draw [thick] (-0.65,0.6) -- (-0.65,-0.6) ;
\draw (0,0.4) node{$\dots$} ;
\fill [white] (0.65,0) circle (0.1) ;
\draw [thick] (0.65,0.6) -- (0.65,-0.6) ;
\draw (0,0.85) node{$\overbrace{{\color{white}..............}}$} ;
\draw (0,1.15) node{$g_i=w_{J_i}$} ;
\end{tikzpicture}
\end{center}

Let $D_i^{p_i}$ (resp. $D_i^{p_i'}$) be the subinterval $D_i\setminus\{p_i\}$ (resp. $D_i\setminus\{p_i'\}$) of $D_i$, and $k\in\ZZ$ such that $\ov{\ell}_i=\ov{m}_i^{-k}w_{D_i^{p_i}}$. Then $w\!_{D_i^{p_i'}}=(w_{D_i^{p_i}})^{g_i}$ and $a_i^{\ast}(w\!_{D_i^{p_i'}})=a_i^{\ast}(w_{D_i^{p_i}})=k$, so $\ov{\ell}_i^{g_i}=(\ov{m}_i^{-k})^{g_i}(w_{D_i^{p_i}})^{g_i}=\ov{m}_i'^{-k}w\!_{D_i^{p_i'}}=\ov{\ell}_i'$. As a result, the $q$--nilpotent peripheral systems $(N_qG_D,\ov{m},\ov{\ell})$ and $(N_qG_D,\ov{m}',\ov{\ell}')$ are equivalent by the isomorphism $\varphi=\text{id}_{N_qG_D}$, with conjugating elements $g_i=w_{J_i}$.
\end{proof}

The following proposition shows that the equivalence class of $q$--nilpotent peripheral system is an invariant of welded links up to $w_q$--equivalence.

\begin{prop}\label{qnpsinvRwq}
If two virtual diagrams differ by a Reidemeister move or by a surgery along a w-tree of degree $\geq q$, then they have equivalent $q$--nilpotent peripheral systems. Moreover, if the points on the diagrams associated to the meridians stay fixed and outside of the local move, then the isomorphism between the $q$--nilpotent peripheral systems is such that $g_i=1$ for $1\leq i\leq n$.
\end{prop}

\begin{proof}
If two virtual diagrams $D$ and $D'$ differ by a Reidemeister move, there is an isomorphism $\wt{\varphi}:G_D\to G_{D'}$ sending each generator associated to an arc of $D$ outside the local move to the corresponding one in $D'$, inducing an isomorphism $\varphi:G=N_q G_D\to G'=N_q G_{D'}$. If the points associated to the meridians stay fixed outside of the local move, then we have $\varphi(\ov{m}_i)=\ov{m}_i'$, and it is easily checked for each Reidemeister move that $\varphi(\ov{\ell}_i)=\ov{\ell}_i'$. If some points associated to the meridians are inside the local move, we change their positions before and after the move to get to the previous case, which does not change the equivalence class of the $q$--nilpotent peripheral systems by Proposition~\ref{equivqnps}. \\

If $D'$ is obtained from $D$ by a surgery along a w-tree of degree $k\geq q$, the presentation of $G_{D'}$ differs from the presentation of $G_D$ by duplicating the generator associated to the arc of $D$ containing the head of the w-tree, and adding a Wirtinger relation between them, whose conjugating element is the commutator of weight $k$ obtained by propagating generators along the w-tree. Since $k\geq q$, this conjugating element is trivial in $N_qG_{D'}$, so the added Wirtinger relation identifies the duplicated generators. As a result, the map $\varphi:N_qG_D\to N_qG_{D'}$ sending each generator of $G_D$ to the corresponding one in $G_{D'}$ is an isomorphism, which gives an equivalence of $q$--nilpotent peripheral systems for which $g_i=1$ for $1\leq i\leq n$.
\end{proof}

Since the $q$--nilpotent peripheral system of a welded string link is equivalent to its Milnor invariants of length $\leq q$, by Proposition~\ref{invconc} it is an invariant of welded concordance. The same is true for the (equivalence class of) $q$--nilpotent peripheral system on welded links. This follows from Theorem 5.5 in~\cite{MCIKSB}, as it is easily checked that the existence of a welded concordance between two welded links implies that of a cut-concordance, as was the case for welded string links.

\begin{prop}\label{qnpsinvwC}{\cite[Thm. 5.5]{MCIKSB}}
If two virtual diagrams differ by a welded concordance, then they have equivalent $q$--nilpotent peripheral systems for all $q\geq 1$.
\end{prop}

As a result, the equivalence class of $q$--nilpotent peripheral system is an invariant of $w_q$--concordance on welded links.

\subsection{Classification of welded links up to $w_q$--concordance}\label{Sec32}

In this section, we will prove that the equivalence class of $q$-nilpotent peripheral systems is a complete invariant of $w_q$--concordance on welded links. The first step is to show that any tree presentation is $w_q$--equivalent to a sorted one (see Definition~\ref{defascsort}), whose associated group admits a presentation involving only meridians and longitudes. As a by-product, we obtain the Chen--Milnor presentation of $N_qG_D$ for any virtual diagram $D$. We then prove that two sorted tree presentations with equivalent $q$--nilpotent peripheral systems are $w_q$--concordant. This in turn gives a characterization of links with trivial Milnor invariants of length $\leq q$.

\begin{prop}\label{equivsortdiag}
For $q\geq 1$, any tree presentation is $w_q$--equivalent to a sorted presentation.
\end{prop}

\begin{proof}
The proof is similar to that of Corollary~\ref{equivascdiag}. The portion of strand $P_i$ which will contain the tails on the $i^{th}$ component of the sorted tree presentation can be chosen arbitrarily, starting as a small neighborhood of a point containing no head. We need to push the tails on this component toward $P_i$, with the notable difference from the case of string links being that we can choose the direction in which we push each tail. As a result, we can avoid having to permute a tail with the head of its w-tree, simply by pushing it in the direction opposite to its head if they are on the same component.

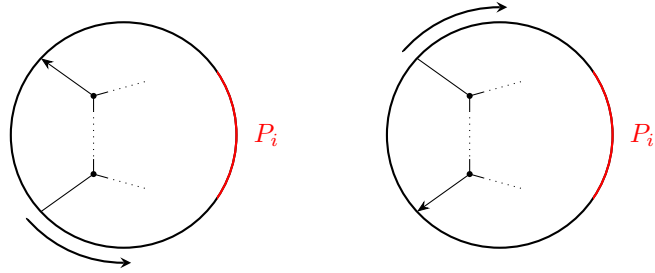
\begin{center}
\begin{tikzpicture}
\draw [thick] (0,0) circle (1.5) ;
\begin{scope}
\clip (1.25,-1.5) -- (1.6,-1.5) -- (1.6,1.5) -- (1.25,1.5) -- cycle ;
\draw [thick,red] (0,0) circle (1.5) ;
\end{scope}
\draw [red] (1.9,0) node{$P_i$} ;
\draw [-Stealth] (-0.4,0.52) -- (-1.1,1.02) ;
\fill [black] (-0.4,0.52) circle (0.04) ;
\draw (-0.4,0.32) -- (-0.4,0.52) -- (-0.22,0.57) ;
\draw [dotted] (-0.22,0.57) -- (0.32,0.72) ;
\draw [dotted] (-0.4,0.32) -- (-0.4,-0.32) ;
\draw (-0.4,-0.32) -- (-0.4,-0.52) ;
\fill [black] (-0.4,-0.52) circle (0.04) ;
\draw (-1.1,-1.02) -- (-0.4,-0.52) -- (-0.22,-0.57) ;
\draw [dotted] (-0.22,-0.57) -- (0.32,-0.72) ;
\begin{scope}
\clip (-1.3,-1.1) -- (0,-1.65) -- (0,-1.75) -- (-1.3,-1.75) -- cycle ;
\draw [thick] (0,0) circle (1.7) ;
\end{scope}
\draw [thick,-stealth] (0,-1.7) -- (0.1,-1.7) ;

\begin{scope}[xshift=5cm]
\draw [thick] (0,0) circle (1.5) ;
\begin{scope}
\clip (1.25,-1.5) -- (1.6,-1.5) -- (1.6,1.5) -- (1.25,1.5) -- cycle ;
\draw [thick,red] (0,0) circle (1.5) ;
\end{scope}
\draw [red] (1.9,0) node{$P_i$} ;
\draw (-0.4,0.52) -- (-1.1,1.02) ;
\fill [black] (-0.4,0.52) circle (0.04) ;
\draw (-0.4,0.32) -- (-0.4,0.52) -- (-0.22,0.57) ;
\draw [dotted] (-0.22,0.57) -- (0.32,0.72) ;
\draw [dotted] (-0.4,0.32) -- (-0.4,-0.32) ;
\draw (-0.4,-0.32) -- (-0.4,-0.52) ;
\fill [black] (-0.4,-0.52) circle (0.04) ;
\draw (-0.4,-0.52) -- (-0.22,-0.57) ;
\draw [-Stealth] (-0.4,-0.52) -- (-1.1,-1.02) ;
\draw [dotted] (-0.22,-0.57) -- (0.32,-0.72) ;
\begin{scope}
\clip (-1.3,1.1) -- (0,1.65) -- (0,1.75) -- (-1.3,1.75) -- cycle ;
\draw [thick] (0,0) circle (1.7) ;
\end{scope}
\draw [thick,-stealth] (0,1.7) -- (0.1,1.7) ;
\end{scope}
\end{tikzpicture}
\end{center}

Since the permutation of a tail with its head was the only move requiring a welded concordance in the proof of Corollary~\ref{equivascdiag}, we only need moves which can be achieved by $w_q$--equivalence in order to get to a sorted presentation. This can be done using the same induction process as in the aforementioned proof, first on the number of tails outside $P_i$, and then for each tail on the number of heads in between that tail and $P_i$ in the direction which we choose to push the tail.
\end{proof}

An important detail to note in the proof of Proposition~\ref{equivsortdiag} is that at no point in the process does any head enter the portion of strand $P_i$. By Proposition~\ref{qnpsinvRwq}, if we choose the meridian $m_i$ in $P_i$, then the isomorphism relating the $q$--nilpotent peripheral system of the original diagram to that of the sorted diagram is such that $g_i=1$. \\

If $(V,\T)$ is a sorted presentation of a virtual diagram $D$ with meridians on the arcs containing the tails, every generator can be expressed in terms of the meridians, and the only relations left are the commutation of each meridian with its longitude, so $G_D\simeq\gp{m_1,\ldots,m_n\,|\,[m_1,\ell_1],\ldots,[m_n,\ell_n]}$. By Proposition~\ref{gencom}, we obtain a presentation of $N_qG_D$ by adding the commutators of weight $\geq q$ on generators $m_i^{\pm 1}$ in the relations. Since any tree presentation is $w_q$--equivalent to a sorted one, and $w_q$--equivalent diagrams have isomorphic $q$--nilpotent groups, we obtain the following Chen--Milnor presentation, which was established by J. Milnor for classical links in~\cite{Mil57}:

\begin{prop}
For any virtual diagram $D$ and $q\geq 1$, we have:
$$N_q G_D\simeq\gp{m_1,\ldots,m_n\,|\,C_q,[m_1,\ov{\ell}_1],\ldots,[m_n,\ov{\ell}_n]},$$
where the $m_i$'s are meridians of $D$, the $\ov{\ell}_i$'s are the associated longitudes expressed in terms of the meridians and $C_q$ is the set of commutators of weight $\geq q$ on the generators $m_i^{\pm 1}$.
\end{prop}

Similar Chen--Milnor type theorems on welded string links and welded links were already established by M. Chrisman in~\cite{MCICHR}, and by B. Audoux, J-B. Meilhan and A. Yasuhara in~\cite{MCIKSB} using the notion of cut diagrams.

\begin{prop}\label{classsortdiag}
Two sorted tree presentations which have equivalent $q$--nilpotent peripheral systems are $w_q$--concordant.
\end{prop}

We now prove that the equivalence class of $q$--nilpotent peripheral systems classifies welded links up to $w_q$--concordance, starting with the special case of sorted presentations.

\begin{proof}
By expanding the w-trees, we can assume that the sorted presentations are arrow presentations, denoted $(V,\A)$ and $(V,\A')$ respectively. Let $D$ and $D'$ be the virtual diagrams obtained from these presentations. By Proposition~\ref{equivsortdiag}, we can assume the points $p_i$ (resp. $p_i'$) associated to the meridians of the presentation $(V,\A)$ (resp. $(V,\A')$) are placed on the arcs containing the tails. We first deal with the case of an isomorphism $\varphi:N_qG_D\to N_qG_{D'}$ such that $g_i=1$ for $1\leq i\leq n$. Using the same notation as in the previous section, and performing some isolated arrow moves if necessary, we can assume that the words $l_i:=w_{D_i^{p_i}}$ (resp. $l_i':=w\!_{D_i'^{p_i'}}$) satisfy $a_i^{\ast}(l_i)=0$ (resp. $a_i'^{\ast}(l_i)=0$). Here the $l_i$'s (resp. $l_i'$'s) are considered as words on the $m_j^{\pm 1}$ (resp. $m_j'^{\pm 1}$'s) without any relation between the generators. In order to compare $l_i$ and $l_i'$, we introduce the word $\tilde{l}_i'$ on the $m_j^{\pm 1}$ obtained by replacing $m_j'$ by $m_j$ in $l_i'$. Then the words $l_i$ and $\tilde{l}_i'$ both represent the longitude $\ov{\ell}_i=\varphi^{-1}(\ov{\ell}_i')\in N_qG_D$. As a result, by the presentation of $N_qG_D$ given above, we can go from $l_i$ to $\tilde{l}_i'$ by adding/deleting occurences of $m_jm_j^{-1}$, $m_j^{-1}m_j$, commutators of weight $\geq q$ on the $m_j$'s, and $[m_j,\ov{\ell}_j]^{\pm 1}$. \\

For each $1\leq i\leq n$ and each addition/deletion of $[m_j,\ov{\ell}_j]^{\pm 1}$ in $\ell_i$, we preemptively create one closed component, which will be attached to the $i^{th}$ component of $(V,\A)$. By an inverse move, we add a pair of cancelling w-arrows from the arc of $V_j$ containing the meridian $m_j$ to the new closed component, then we slide the tail of the non-twisted arrow along $V_j$ with the reverse orientation, using inverse and slide moves to cross the heads (see Figure~\ref{FigcrossR2R3}). This has the effect of creating the word $[m_j,\ov{\ell}_j]$ on the new component. To create the word $[m_j,\ov{\ell}_j]^{-1}=[\ov{\ell}_j,m_j]$, we slide the tail of the twisted arrow instead of the non-twisted one.

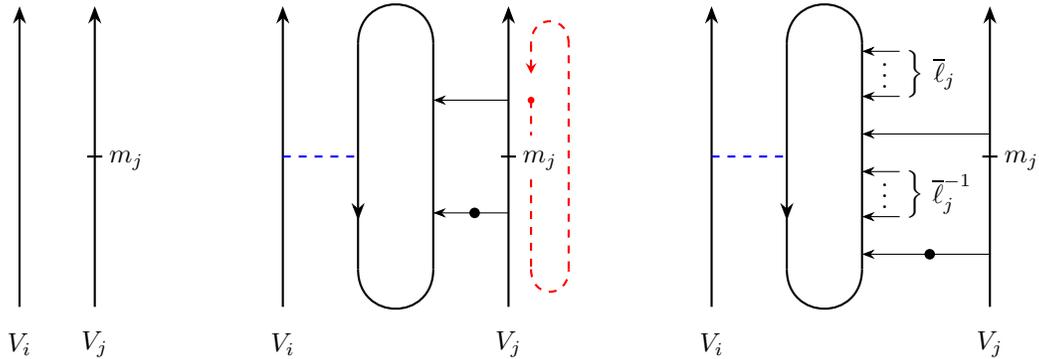
\begin{center}
\begin{tikzpicture}
\begin{scope}[xshift=0cm]
\draw [thick,-Stealth] (0,0) -- (0,4) ;
\draw (0,-0.5) node{$V_i$} ;
\draw [thick,-Stealth] (1,0) -- (1,4) ;
\draw (1,-0.5) node{$V_j$} ;
\draw [thick] (0.9,2) -- (1.1,2) ;
\draw (1.42,1.95) node{$m_j$} ;
\end{scope}

\begin{scope}[xshift=3.5cm]
\draw [thick,-Stealth] (0,0) -- (0,4) ;
\draw (0,-0.5) node{$V_i$} ;
\draw [thick] (1,0.5) -- (1,3.5) ;
\draw [thick] (1,3.5) .. controls +(0,0.7) and +(0,0.7) .. (2,3.5) ;
\draw [thick] (1,0.5) .. controls +(0,-0.7) and +(0,-0.7) .. (2,0.5) ;
\draw [thick] (2,0.5) -- (2,3.5) ;
\draw [thick,-Stealth] (1,1.25) -- (1,1.15) ;
\draw [thick,dashed,blue] (0,2) -- (1,2) ;
\draw [thick,-Stealth] (3,0) -- (3,4) ;
\draw (3,-0.5) node{$V_j$} ;
\draw [thick] (2.9,2) -- (3.1,2) ;
\fill [red] (3.3,2.75) circle (0.045) ;
\draw [thick,dashed,red] (3.3,2.75) -- (3.3,0.5) ;
\draw [thick,dashed,red] (3.3,0.5) .. controls +(0,-0.4) and +(0,-0.4) .. (3.8,0.5) ;
\draw [thick,dashed,red] (3.8,0.5) -- (3.8,3.5) ;
\draw [thick,dashed,red] (3.8,3.5) .. controls +(0,0.4) and +(0,0.4) .. (3.3,3.5) ;
\draw [thick,dashed,red,-stealth] (3.3,3.5) -- (3.3,3.1) ;
\fill [white] (3.15,1.7) -- (3.65,1.7) -- (3.65,2.28) -- (3.15,2.28) -- cycle ;
\draw (3.42,1.95) node{$m_j$} ;
\draw [-Stealth] (3,2.75) -- (2,2.75) ;
\draw [-Stealth] (3,1.25) -- (2,1.25) ;
\fill [black] (2.55,1.25) circle (0.07) ;
\end{scope}

\begin{scope}[xshift=9.2cm]
\draw [thick,-Stealth] (0,0) -- (0,4) ;
\draw (0,-0.5) node{$V_i$} ;
\draw [thick] (1,0.5) -- (1,3.5) ;
\draw [thick] (1,3.5) .. controls +(0,0.7) and +(0,0.7) .. (2,3.5) ;
\draw [thick] (1,0.5) .. controls +(0,-0.7) and +(0,-0.7) .. (2,0.5) ;
\draw [thick] (2,0.5) -- (2,3.5) ;
\draw [thick,-Stealth] (1,1.25) -- (1,1.15) ;
\draw [thick,dashed,blue] (0,2) -- (1,2) ;
\draw [thick,-Stealth] (3.7,0) -- (3.7,4) ;
\draw (3.7,-0.5) node{$V_j$} ;
\draw [thick] (3.6,2) -- (3.8,2) ;
\draw (4.12,1.95) node{$m_j$} ;
\draw [-Stealth] (3.7,2.3) -- (2,2.3) ;
\draw [-Stealth] (2.5,2.8) -- (2,2.8) ;
\draw (2.3,3.18) node{$\vdots$} ;
\draw [-Stealth] (2.5,3.4) -- (2,3.4) ;
\draw (2.7,3.1) node{$\Big{\}}$} ;
\draw (3.1,3.1) node{$\ov{\ell}_j$} ;
\draw [-Stealth] (2.5,1.8) -- (2,1.8) ;
\draw (2.3,1.58) node{$\vdots$} ;
\draw [-Stealth] (2.5,1.2) -- (2,1.2) ;
\draw (2.7,1.5) node{$\Big{\}}$} ;
\draw (3.2,1.5) node{$\ov{\ell}_j^{-1}$} ;
\draw [-Stealth] (3.7,0.7) -- (2,0.7) ;
\fill [black] (2.9,0.7) circle (0.07) ;
\end{scope}
\end{tikzpicture}
\end{center}

We need to create these extra closed components in order to make sure that the words $\ell_i$ representing the longitudes are not modified before all the required words $[m_j,\ov{\ell}_j]^{\pm 1}$ have been created. Moreover, trying to insert the word $[m_i,\ov{\ell}_i]^{\pm 1}$ directly into the $i^{th}$ component without creating an extra one would result in a neverending loop as new w-arrows would constantly be created when trying to slide the tail of the w-arrow illustrated above through its head. \\

Once these extra components have been added, we can start to implement the necessary modifications on $(V,\A)$ to go from $l_i$ to $\tilde{l}_i'$: an addition/deletion of $m_jm_j^{-1}$ or $m_j^{-1}m_j$ can be achieved by an inverse move, the addition/deletion of a commutator of weight $\geq q$ on the $m_j^{\pm 1}$ can be achieved by $w_q$--equivalence, and the addition of $[m_j,\ov{\ell}_j]^{\pm 1}$ in the word $l_i$ can be achieved by attaching one of the extra components containing the word $[m_j,\ov{\ell}_j]^{\pm 1}$ to $V_i$ on the appropriate arc by a saddle move. In order to delete such a commutator, we insert its inverse by the same process, then cancel out the product by inverse moves. As a result, we can modify $(V,\A)$ into a $w_q$--concordant tree presentation $(V,\T'')$ whose longitudes are represented by the same words as those of $(V,\A')$. After expanding the w-trees, we hence obtain an arrow presentation equal to $(V,\A')$, and as a result $(V,\A)$ and $(V,\A')$ are $w_q$--concordant. \\

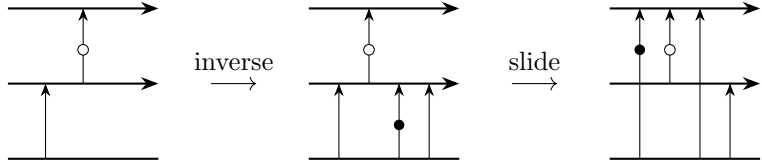
\begin{figure}
\centering
\begin{tikzpicture}
\draw [thick,-Stealth] (0,2) -- (2,2) ;
\draw [thick,-Stealth] (0,1) -- (2,1) ;
\draw [thick] (0,0) -- (2,0) ;
\draw [-Stealth] (0.5,0) -- (0.5,1) ;
\draw [-Stealth] (1,1) -- (1,2) ;
\draw [fill=white] (1,1.45) circle (0.07) ;

\draw (3,1.3) node{inverse} ;
\draw [->] (2.7,1) -- (3.3,1) ;

\begin{scope}[xshift=4cm]
\draw [thick,-Stealth] (0,2) -- (2,2) ;
\draw [thick,-Stealth] (0,1) -- (2,1) ;
\draw [thick] (0,0) -- (2,0) ;
\draw [-Stealth] (0.4,0) -- (0.4,1) ;
\draw [-Stealth] (0.8,1) -- (0.8,2) ;
\draw [fill=white] (0.8,1.45) circle (0.07) ;
\draw [-Stealth] (1.2,0) -- (1.2,1) ;
\fill [black] (1.2,0.45) circle (0.07) ;
\draw [-Stealth] (1.6,0) -- (1.6,1) ;

\draw (3,1.3) node{slide} ;
\draw [->] (2.7,1) -- (3.3,1) ;
\end{scope}

\begin{scope}[xshift=8cm]
\draw [thick,-Stealth] (0,2) -- (2,2) ;
\draw [thick,-Stealth] (0,1) -- (2,1) ;
\draw [thick] (0,0) -- (2,0) ;
\draw [-Stealth] (0.4,0) -- (0.4,2) ;
\fill [black] (0.4,1.45) circle (0.07) ;
\draw [-Stealth] (0.8,1) -- (0.8,2) ;
\draw [fill=white] (0.8,1.45) circle (0.07) ;
\draw [-Stealth] (1.2,0) -- (1.2,2) ;
\draw [-Stealth] (1.6,0) -- (1.6,1) ;
\end{scope}
\end{tikzpicture}
\caption{Crossing a head using inverse and slide moves.}\label{FigcrossR2R3}
\end{figure}

We now deal with the general case of an isomorphism $\varphi$ with potentially non trivial conjugating elements $g_i\in N_qG_D$. We will modify $(V,\A)$ in order to get to the configuration above. Using inverse moves, we add w-arrow heads representing the word $g_1$ followed by the word $g_1^{-1}$ next to the arc containing the tails on $V_1$. As in the proof of Proposition~\ref{equivsortdiag}, we can slide the tails on $D_1$ to the portion of strand between the words $g_1$ and $g_1^{-1}$ by using $w_q$--equivalence. Moreover, we can preserve the sorted structure of the other components, and keep all the heads outside of the portion of strand between the words $g_1$ and $g_1^{-1}$. We obtain a sorted presentation $(V,\T'')$ of a virtual diagram $D''$ which is $w_q$--equivalent to $D$, with a $q$--nilpotent peripheral system $(N_qG_{D''},m'',\ov{\ell}'')$ such that each meridian $m_i''$ is placed on the arc containing the tails. By the proof of Proposition~\ref{equivqnps} and by Proposition~\ref{qnpsinvRwq}, the $q$--nilpotent peripheral system $(N_qG_{D''},m'',\ov{\ell}'')$ is equivalent to $(N_qG_D,m,\ov{\ell})$ by an isomorphism $\psi:N_qG_{D''}\to N_qG_D$ such that $\psi(m_1'')=m_1^{g_1}$, $\psi(\ov{\ell}_1'')=\ov{\ell}_1^{g_1}$ and $\psi(m_i'')=m_i$, $\psi(\ov{\ell}_i'')=\ov{\ell}_i$ for $i\geq 2$. The isomorphism $\varphi\circ\psi:N_qG_{D''}\to N_qG_{D'}$ sends $m_1''$ to $m_1'$, $\ov{\ell}_1''$ to $\ov{\ell}_1'$, and $(m_i'')^{\psi^{-1}(g_i)}$ to $m_i'$, $(\ov{\ell}_i'')^{\psi^{-1}(g_i)}$ to $\ov{\ell}_i'$ for $i\geq 2$. By iterating this process for each component, we obtain a sorted tree presentation which is $w_q$--equivalent to $(V,\A)$ and whose $q$--nilpotent peripheral system is equivalent to that of $(V,\A')$ by an isomorphism which sends meridians to meridians and longitudes to longitudes. By the special case above, this concludes the proof of the proposition.
\end{proof}

Note that contrary to the case of welded string links, the notion of welded concordance is not used to obtain a sorted presentation, but is necessary to relate two sorted presentations with the same invariants. Using Propositions~\ref{qnpsinvRwq}, \ref{qnpsinvwC} and \ref{classsortdiag}, we obtain the main result of the section:

\begin{theo}\label{classwldlk}
The equivalence class of $q$--nilpotent peripheral systems is a complete invariant of $w_q$--concordance on welded links.
\end{theo}

The Milnor invariants of a welded link can also be obtained as the coefficients of the power series $E_q(\ov{\ell}_i)$. However, as in the case of classical links, because of the remaining relations in $N_qG_D$ the Milnor invariants of length $q$ are only well defined modulo some of the Milnor invariants of length $<q$, and only the first non-vanishing Milnor invariants of the link are well-defined as integers. We will use Theorem~\ref{classwldlk} to give a characterization of welded links with Milnor invariants of length $\leq q$ equal to zero. This will require a more precise description of the Chen--Milnor presentation, using the Chen maps defined below. \\

For a virtual diagram $D$, let $\gp{a_{ij}}$ be the free group with one generator $a_{ij}$ for each arc of $D_i$. Once the points $p_i$ associated to the meridians $m_i=a_{i0}$ are chosen, we denote by $w_{ij}$ the word $w_J$, where $J$ is the subinterval of $D_i$ from $p_i$ to a point on the arc associated to the generator $a_{ij}$. The group $G_D$ is then equal to $\gp{a_{ij}\,|\,a_{ij}=m_i^{w_{ij}}}$.

\begin{de}\label{chenmaps}
The Chen maps $\eta_q:\gp{a_{ij}}\to F_n=\gp{m_i}$ associated to a virtual diagram $D$ are defined by induction on $q\geq 1$ by setting $\eta_1(a_{ij}):=m_i$ for $1\leq i\leq n$, and $\eta_{q+1}(a_{ij}):=\eta_q\big(m_i^{w_{ij}}\big)$ for $q\geq 1$ and $1\leq i\leq n$.
\end{de}

It is easily checked that $\eta_{q+1}\equiv\eta_q\!\mod\G_qG_D$ for $q\geq 1$. Moreover, the Chen maps induce homomorphisms $\ov{\eta}_q:N_qG_D\to N_q\gp{m_i\,|\,[m_i,\eta_q(\ell_i)]}$, and these are isomorphisms (see \cite{MCICHR}). The Milnor invariants of length $\leq q$ of $D$ can be obtained from the power series $E_q(\ov{\eta}_q(\ov{\ell}_i))$, as the Chen maps give a way to express the longitudes in terms of the meridians.

\begin{cor}\label{classwldunlk}
A welded link is $w_q$--concordant to the unlink if and only if it has vanishing Milnor invariants of length $\leq q$.
\end{cor}

\begin{proof}
Let $L$ be a welded link represented by a virtual diagram $D$. If $L$ is $w_q$--concordant to the unlink, by Theorem~\ref{classwldlk} its $q$--nilpotent peripheral system is equivalent to the $q$--nilpotent peripheral system of the unlink, and in particular the longitudes $\ov{\ell}_i$ are trivial in $N_qG_D$. As a result, $E_q(\ov{\eta}_q(\ov{\ell}_i))=1$ for $1\leq i\leq n$, and the Milnor invariants of length $\leq q$ of $L$ are all zero. \\

Conversely, suppose that the Milnor invariants of length $\leq q$ of $L$ are all zero. We prove by induction on $1\leq k\leq q$ that $\eta_k(\ell_i)\in\G_kF_n$ for $1\leq i\leq n$. For $k=1$ it is trivially true since $\G_1F_n=F_n$. If it is true for $k-1$, then $\eta_{k-1}(\ell_i)\in\G_{k-1}F_n$, and since $\eta_k\equiv\eta_{k-1}\!\mod\G_{k-1}F_n$, we have:
$$[m_i,\eta_k(\ell_i)]\equiv[m_i,\eta_{k-1}(\ell_i)]\equiv 1\!\mod\G_kF_n,$$
hence $[m_i,\eta_k(\ell_i)]\in\G_kF_n$. As a result, the induced Chen map $\ov{\eta}_k$ is:
$$\ov{\eta}_k:N_kG_D\to N_k\gp{m_i\,|\,[m_i,\eta_k(\ell_i)]}\simeq N_kF_n.$$
The Milnor invariants of length $k$ of $L$ are then the coefficients of the monomials of degree $k-1$ of the power series $E_k(\ov{\eta}_k(\ov{\ell}_i))\in U_1(\Sn)/(1+\X^k)$. Since these are zero by our hypothesis on $L$, $E_k(\ov{\eta}_k(\ov{\ell}_i))=1$, and by Proposition~\ref{injMagexp} we have $\ov{\eta}_k(\ov{\ell}_i)=1\in N_kF_n$ for $1\leq i\leq n$. As a result $\eta_k(\ell_i)\in\G_kF_n$, which concludes the induction on $k$. The $q$--nilpotent peripheral system of $L$ is then equivalent to that of the unlink by the induced Chen map $\ov{\eta}_q:N_qG_D\to N_qF_n$, which sends the meridians of $L$ to the generators of $N_qF_n$ and the longitudes of $L$ to $1$. By Theorem~\ref{classwldlk}, $L$ is $w_q$--concordant to the unlink.
\end{proof}

As a result, we have the following diagrammatical characterization of classical links with vanishing Milnor invariants of length $\leq q$:

\begin{cor}\label{classclsunlk}
A classical link has vanishing Milnor invariants of length $\leq q$ if and only if it is $w_q$--concordant to the unlink as a welded link.
\end{cor}

This is to be compared to the result of J. Conant, R. Schneiderman and P. Teichner on the classification of $C_q$--concordance on classical links {\cite[Cor. 3]{ConSchTei}}. As this classification involves the additional higher-order Sato-Levine and Arf invariants, it is to be expected that some classical links might be $w_q$--concordant to the unlink as a welded link but not $C_q$--concordant to it as a classical link. Note that by Theorem 8.1 of \cite{MeiYas}, welded knots are already trivial up to $w_q$--equivalence (and hence up to $w_q$--concordance), and by Theorem 5 of \cite{Gau}, welded long knots (i.e. welded string links with one component) are trivial up to welded concordance, so the case of one component is always trivial up to $w_q$--concordance. We give an example below illustrating the difference between $C_q$--concordance and $w_q$--concordance on a link with $2$ components.

\paragraph{Example:}
Illustrated in Figure~\ref{Figex} is a link which is $w_3$--concordant to the unlink, as it has trivial Milnor invariants of length $\leq 3$, but not $C_3$--concordant to the unlink, as it has non-trivial Sato-Levine invariants of order 2 (see \cite{WTCCL} Section 5.1 for indications on how to compute these invariants). This link is obtained, modulo one $R1$ move, by performing a clasper surgery on the unlink along the tree of order 1 (or degree 2) represented on the left. Due to the self-annihilation relation $[X,X]=0$ in a Lie algebra, the Milnor invariants of length $3$, the first potentially non-trivial ones, vanish. The Sato-Levine invariants are designed to detect such trees, as they take their values in a quasi-Lie algebra where the self-annihilation relation is replaced by the antisymmetry $[Y,X]=-[X,Y]$.

\begin{figure}
\centering
\begin{tikzpicture}
\draw (-0.5,-0.05) circle (0.1) ;
\fill [white] (-0.41,-0.035) circle (0.05) ;
\draw (0.5,-0.05) circle (0.1) ;
\fill [white] (0.59,-0.06) circle (0.05) ;
\draw [thick] (0,-0.75) ellipse (1.5 and 0.75) ;
\draw [thick,-stealth] (0.05,-1.5) -- (-0.05,-1.5) ;
\draw (0,-1.95) node{$1$} ;
\draw (0,1.3) circle (0.1) ;
\fill [white] (-0.1,1.3) circle (0.05) ;
\draw [thick] (0,1.8) ellipse (1 and 0.5) ;
\draw [thick,-stealth] (-0.05,2.3) -- (0.05,2.3) ;
\draw (0,2.7) node{$2$} ;
\draw (-0.45,0.05) -- (0,0.65) -- (0.45,0.05) ;
\fill [white] (-0.59,-0.06) circle (0.05) ;
\begin{scope}
\clip (-0.5,-0.25) -- (-0.7,-0.25) -- (-0.7,0.15) -- (-0.5,0.15) -- cycle ;
\draw (-0.5,-0.05) circle (0.1) ;
\end{scope}
\fill [white] (0.41,-0.035) circle (0.05) ;
\begin{scope}
\clip (0.5,-0.25) -- (0.3,-0.25) -- (0.3,0.15) -- (0.5,0.15) -- cycle ;
\draw (0.5,-0.05) circle (0.1) ;
\end{scope}
\draw (0,0.65) -- (0,1.2) ;
\fill [white] (0.1,1.3) circle (0.05) ;
\begin{scope}
\clip (0,1.1) -- (0,1.5) -- (0.2,1.5) -- (0.2,1.1) -- cycle ;
\draw (0,1.3) circle (0.1) ;
\end{scope}
\fill [black] (0,0.65) circle (0.05) ;
\draw (3,0.65) node{$=$} ;

\begin{scope}[xshift=7cm]
\draw [thick] (0,1.5) ellipse (2 and 1) ;

\fill [white] (1.67,0.97) circle (0.15) ;
\fill [white] (-0.5,0.53) circle (0.15) ;

\draw [thick] (-2,0) .. controls +(0,1) and +(-0.7,0) .. (-0.5,1.5) ;
\draw [thick] (0,-0.3) .. controls +(-0.8,0) and +(-1.2,0) .. (0.5,1.5) ;
\fill [white] (0,1.37) circle (0.15) ;
\draw [thick] (-0.5,1.5) .. controls +(1.2,0) and +(0.8,0) .. (0,-0.3) ;
\draw [thick] (0.5,1.5) .. controls +(0.7,0) and +(0,1) .. (2,0) ;
\draw [thick] (2,0) .. controls +(0,-1) and +(1,0) .. (0,-1.5) ;
\draw [thick] (0,-1.5) .. controls +(-1,0) and +(0,-1) .. (-2,0) ;
\draw [thick,-stealth] (0.05,-1.5) -- (-0.05,-1.5) ;
\draw (0,-1.95) node{$1$} ;

\fill [white] (-1.67,0.97) circle (0.15) ;
\fill [white] (0.5,0.53) circle (0.15) ;

\begin{scope}
\clip (-1.9,1.2) -- (-1.4,1.2) -- (-1.4,0.7) -- (-1.9,0.7) -- cycle ;
\draw [thick] (0,1.5) ellipse (2 and 1) ;
\end{scope}

\begin{scope}
\clip (0.25,0.8) -- (0.75,0.8) -- (0.75,0.3) -- (0.25,0.3) -- cycle ;
\draw [thick] (0,1.5) ellipse (2 and 1) ;
\end{scope}

\draw [thick,-stealth] (-0.05,2.5) -- (0.05,2.5) ;
\draw (0,2.9) node{$2$} ;
\end{scope}
\end{tikzpicture}
\caption{A link which is $w_3$--concordant but not $C_3$--concordant to the unlink}\label{Figex}
\end{figure}
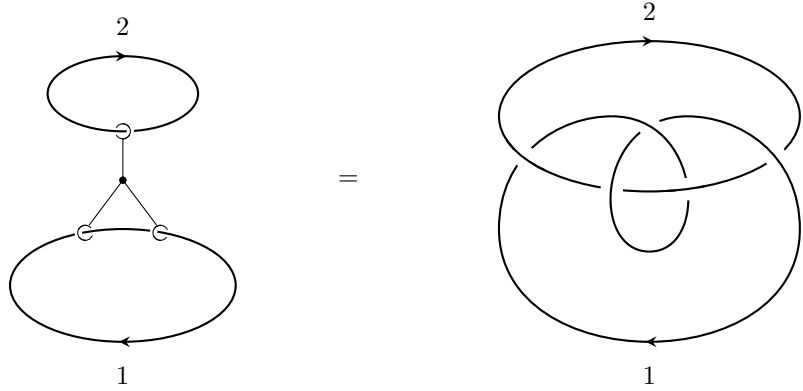

\newpage
\bibliographystyle{plain}
\bibliography{biblio}

\end{document}